\documentclass[a4paper,10pt, draft]{amsart}
\def\date{29.1.2011}  

\usepackage{amsmath, amssymb}
\usepackage{latexsym}
\usepackage[all]{xy}
\usepackage{enumerate}
\usepackage{mescommandes}

\renewcommand{\ph}{\varphi} 
\renewcommand{\phi}{\varphi} 
\renewcommand{\Exp}{\exp} 
\newcommand{\shalf}{{\textstyle{\frac{1}{2}}}}
\newcommand{\derat}[1]{\frac{d}{dt} \hbox{\vrule width0.5pt
                height 5mm depth 3mm${{}\atop{{}\atop{\scriptstyle t=#1}}}$}}

\renewcommand{\1}{{\bf 1}} 
\renewcommand{\Re}{\mathop{{\rm Re}}\nolimits}
\renewcommand{\Im}{\mathop{{\rm Im}}\nolimits}
\newcommand{\eset}{\emptyset} 
\newcommand{\R}{\bR} 
 
\renewcommand{\C}{\mathbb C} 
\newcommand{\Z}{\mathbb Z} 
\newcommand{\cA}{\mathcal A} 
 
\newcommand{\cD}{\mathcal D} 
\newcommand{\cU}{\mathcal U} 
\newcommand{\ran}{\rangle} 
\newcommand{\lan}{\langle} 
\newcommand{\oline}{\overline} 
\renewcommand{\:}{\colon}
\newcommand{\subeq}{\subseteq}
\newcommand{\dd}{{\tt d}} 
\newcommand{\g}{\fg} 
\newcommand{\fu}{\mathfrak{u}}
\newcommand{\su}{\mathfrak{su}}
\newcommand{\fsl}{\mathfrak{sl}}
\newcommand{\so}{\mathfrak{so}}

\newcommand{\Spann}{\mathop{{\rm span}}\nolimits}

\newcommand{\SO}{\mathop{{\rm SO{}}}\nolimits}
\renewcommand{\U}{\mathop{{\rm U{}}}\nolimits}
\newcommand{\OO}{\mathop{{\rm O{}}}\nolimits}
\newcommand{\SU}{\mathop{{\rm SU{}}}\nolimits}

\newcommand{\mlabel}{\label} 
\newcommand{\res}{\vert}
\renewcommand{\hat}{\widehat} 
\renewcommand{\tilde}{\widetilde} 

\newcommand{\Sq}{S_{\bullet}}

\newcommand{\qq}{q_{\bullet}}
\newcommand{\boh}{\boldsymbol{1}_H}
\newcommand{\bo}{\boldsymbol{1}}

\newcommand{\qg}{q_{\Gamma}}

\DeclareMathOperator{\Spec}{Spec}

\newcommand{\fgc}{\fg_c}
\newcommand{\fgC}{\fg_\bC}

\newcommand{\cHy}{\cH^\infty}
\newcommand{\cHk}{\cH_K}
\newcommand{\cHr}{\cH_K^0}

\newcommand{\dif}[1]{\frac{d}{d#1}}
\newcommand{\df}[1]{\frac{d}{d#1}\Bigr|_{#1=0}}

\newcommand{\dpi}{\mathrm{d}\pi}
\newcommand{\dro}{\mathrm{d}\rho}
\newcommand{\odro}{\ov{\mathrm{d}\rho}}
\newcommand{\odpi}{\ov{\mathrm{d}\pi}}

\newcommand{\eidpi}[1]{e^{i\ov{\mathrm{d}\pi}(#1)}}

\newcommand{\interior}{\mathrm{int}}

\newtheorem{theo}{Theorem}
\newtheorem{pro}[theo]{Proposition}
\newtheorem{lem}[theo]{Lemma}
\newtheorem{cor}[theo]{Corollary}

\numberwithin{theo}{section} %

\theoremstyle{definition}
\newtheorem{de}[theo]{Definition}
\newtheorem{rem}[theo]{Remark}
 
\newtheorem{exs}[theo]{Examples} 

\newenvironment{prf}{\begin{proof}}{\end{proof}}

\title[Analytic extension techniques]
{Analytic extension techniques for\\ unitary representations 
of Banach--Lie groups} 
\author{St\'ephane Merigon and Karl-Hermann Neeb}
\address{Department  Mathematik, FAU Erlangen-N\"urnberg, 
Bismarckstrasse 1 1/2, 91054-Erlangen, Germany; neeb@mi.uni-erlangen.de}

\begin{document}

\maketitle


\begin{abstract} 
Let $(G,\theta)$ be a Banach--Lie group with involutive 
automorphism $\theta$, $\g = \fh \oplus \fq$ be the $\theta$-eigenspaces 
in the Lie algebra $\g$ of $G$, and $H = (G^\theta)_0$ be the identity 
component of its group of fixed points. An 
Olshanski semigroup is a semigroup $S \subeq G$ of the form 
$S = H \exp(W)$, where $W$ is an open $\Ad(H)$-invariant convex 
cone in $\fq$ and the polar map 
$H \times W \to S, (h,x) \mapsto h \exp x$ is a diffeomorphism. 
Any such semigroup carries an involution $*$ satisfying 
$(h\exp x)^* = (\exp x) h^{-1}$. Our central result, 
generalizing the L\"uscher--Mack Theorem for finite dimensional groups, 
asserts 
that any locally bounded $*$-representation $\pi \: S \to B(\cH)$ 
with a dense set of smooth vectors defines by ``analytic continuation''  
a unitary representation of the simply connected Lie group 
$G_c$ with Lie algebra $ \g_c = \fh + i \fq$. 
We also characterize those unitary representations of $G_c$ 
obtained by this  construction. With similar methods, we further 
show that semibounded unitary representations extend to 
holomorphic representations of complex Olshanski semigroups.   
\end{abstract}

\section{Introduction}

There are many important results in the unitary representation 
theory of Lie groups related to analytic continuation. 
Here a key ingredient is the special case
 where $\pi \: \R\to \U(\cH)$ is a strongly continuous unitary 
one-parameter group and $A$ its self-adjoint infinitesimal 
generator, i.e., $\pi(t) = e^{itA}$ in the sense of measurable 
functional calculus. 
Then the unitary one-parameter group $\pi$ extends to a holomorphic 
one-parameter semigroup $\hat\pi \: \bC_+ = \R + i \R_{\geq 0} \to B(\cH)$ 
if and only if $A$ is bounded below. This extension then restricts 
to a locally bounded non-degenerate hermitian one-parameter 
semigroup $t \mapsto \hat\pi(it)$. Conversely, any such hermitian 
one-parameter group has a self-adjoint infinitesimal generator 
$-A$ and $e^{izA}$ then yields an extension to $\bC_+$, where the boundary 
values form a unitary one-parameter group. The key point of this 
picture is that self-adjoint operators $A$ whose spectrum 
is bounded below can be viewed as the infinitesimal generators of two 
objects: a unitary one-parameter group (Stone's Theorem) 
and a hermitian one-parameter semigroup of bounded operators 
(Hille--Yosida Theorem). 
This is the one-parameter context of what we are dealing with 
in the present paper for Banach--Lie groups. 

We call a pair $(G,\theta)$ consisting of a Banach--Lie group 
$G$ and an involutive automorphism $\theta$ of $G$ 
a {\it symmetric Banach--Lie group}. 
We decompose its Lie algebra $\g =  \fh \oplus \fq$ 
into $\pm 1$-eigenspaces of $\theta$  
and write $H := G^\theta_0$ for the identity component of the 
group of $\theta$-fixed points in $G$. 
Let $W \subeq \fq$ be an open convex $\Ad(H)$-invariant cone 
for which the polar map 
\[ H \times W \to H \exp(W), \quad (h,x) \mapsto h \exp x \] 
is a diffeomorphism onto an open subsemigroup $S := S_H(W) 
:= H \exp(W)$ of $G$. 
Then $S$ is an involutive semigroup with respect to the involution 
$s^* := \theta(s)^{-1}$ which is called an {\it Olshanski semigroup}. 
In Appendix~\ref{A} we explain how to obtain such semigroups 
$S_{H'}(W)$, where $H'$ is a connected Lie group locally isomorphic 
to $H$ for which $H'$ and $S_{H'}(W)$ need not be contained in a 
Lie group (cf.\ \cite[Ex.~II.13]{Ne92}). 

Our first main result (Theorem~\ref{thm:3.6}, proved in 
Section~\ref{sec:4}) is 
the following Banach version of the L\"uscher--Mack Theorem 
(\cite{LM75}): Let $\rho:S = S_H(W)\ra B(\cH)$ be a non-degenerate 
strongly continuous $*$-representation of $S$ which is {\it smooth} 
in the sense 
that the space $\cH^\infty$ of smooth vectors is dense. 
If $G_c$ is the simply connected Lie group with Lie algebra 
$\g_c := \fh + i \fq$, then there exists a smooth unitary 
representation $(\pi, \cH)$ of $G_c$ on $\cH$ which is uniquely determined 
by the requirement that the unitary one-parameter groups 
corresponding to elements $x \in \fh$ are those obtained 
by ``extending'' $\rho$ to $H$ and 
the generators of the one-parameter groups corresponding to elements of 
$iW \subeq i\fq$ are of the form $i\dd\rho(x)$, where 
$\dd\rho(x)$ is the infinitesimal generator of the hermitian 
one-parameter semigroup $t \mapsto \rho(\exp tx)$. 
For finite dimensional groups this result is due to 
M.~L\"uscher and G.~Mack (\cite{LM75}; see also \cite{HILGERTNEEB}). 
As their methods make 
heavy use of coordinates obtained from products of one-parameter 
groups, we have to develop a completely new approach in the 
Banach context. Actually our approach is more 
direct and uses only quite general methods, such as 
the criteria for the integrability of infinitesimal unitary 
representations of Banach--Lie algebras 
from \cite{Mer10}. The one-parameter case 
described above corresponds to $G = \R$, $\theta(g) = -g$, 
$S = \R_{> 0}$ and $G_c = i \R$. 

An interesting, much less involved, special 
case arises if $S = G$ is the whole group. 
Then our assumption is that $G$ has a diffeomorphic 
polar decomposition $H \exp \fq$ and our theorem establishes a correspondence 
between $*$-representations $\pi \: G \to \GL(\cH)$ by bounded 
operators and norm-continuous unitary representations of 
$G_c$. If $G$ is a finite dimensional semisimple Lie group 
and $\theta$ a Cartan involution, then 
this is Weyl's well-known unitary  trick 
relating finite dimensional representations of 
$G$ to unitary representations of the compact group~$G_c$. 

In the context of the L\"uscher--Mack Theorem, it is a natural 
question which unitary representations of $G_c$ are obtained 
from representations of a semigroup $S_H(W)$. 
To answer this and related questions, we consider for a smooth 
unitary representation $(\pi, \cH)$ of a Lie group $G$ 
the convex function 
\[ s_\pi \: \g \to \R \cup \{\infty\}, \quad 
s_\pi(x) := \sup\big(\Spec(i\dd\pi(x))\big).\] 
We call $\pi$ {\it semibounded} if $s_\pi$ is bounded in the 
neighborhood of some point $x_0 \in \g$, and if 
$W \subeq \g$ is a convex cone, then we say that 
$\pi$ is {\it $W$-semibounded} if $s_\pi$ is locally bounded on 
$W$. Now the converse to the L\"uscher--Mack Theorem 
(Corollary~\ref{C:conv}) asserts that a representation 
of $G_c$ is obtained from a smooth representation of $S$ if and only if 
it is $iW$-semibounded. Here the main difficulty 
is to show that, for an $iW$-semibounded unitary representation of 
$G_c$, the prescription 
$\rho(h\exp x) := \pi(h) e^{-\dd\pi(x)}$ define a representation 
and to see that it is actually smooth. Here 
we use methods developed previously 
in \cite{Mer10} and \cite{Ne10b}. 

With the same tools we obtain the following Holomorphic 
Extension Theorem (cf.\ \cite{Ol82} and \cite{Ne00} for the finite dimensional 
case): For every open invariant 
cone $W \subeq \g$ and a corresponding complex Olshanski semigroup 
$S_G(iW)$, each $W$-semibounded unitary representation 
of $G$ extends to a holomorphic representation of 
$S_G(iW)$. 

In the finite dimensional context Lawson's Theorem 
guarantees the existence of 
Olshanski semigroups $S_H(W)$ under quite simple 
requirements on the spectra of the operators $\ad x$, $x \in W$ 
(cf.\ \cite{La94}, \cite{Ne00}). Since Lawson's arguments involve 
local compactness in a crucial way, they do not generalize 
to the Banach context. However, natural examples 
of Olshanski semigroups arise as 
compression semigroups of bounded symmetric complex domains in 
Banach spaces, symmetric Hilbert domains, and real forms of 
such domains (cf.\ \cite{Ne01}). 

As we show in \cite{MN11}, the real forms $\cD$ of 
symmetric Hilbert domains $\cD_c$, i.e., the fixed points for an 
antiholomorphic involution $\sigma$,  are of particular interest in 
representation theory (see \cite{Ka83, Ka97} for a classification). 
Here $G_c$ is  a central extension of the identity component 
of $\Aut(\cD_c)$, which leads to the class of {\it hermitian Lie groups} 
whose semibounded representations are classified in \cite{Ne10c}. 
Choosing a base point in $\cD$ 
leads to an involution $\theta$ on $G_c$ with 
$\cD_c \cong G_c/G_c^\theta$, 
and conjugation with $\sigma$ induces another involution on $G_c$. 
The complex domain $\cD_c$ has a natural 
compression semigroup of form 
$S_c = G_c \exp(W_c)$, where $W_c \subeq i\g_c$ is an open 
invariant cone (\cite{Ne01}). 
The subgroup $H := (G_c^\sigma)_0$ is invariant under $\theta$ and 
$\cD \cong H/H^\theta$ is a real symmetric space. 
The corresponding semigroup is $S_H(W)$ for $W = W_c^\sigma$. 
Here our L\"uscher--Mack Theorem provides a bridge between 
$*$-representations of $S_H(W)$ and unitary representations of the group 
$G_c$. A remarkable feature of the infinite dimensional 
context is that a large 
class of the irreducible separable 
continuous unitary representations of 
$H$ extend to contraction representations of $S_H(W)$ 
with the same commutant 
and hence further to representations of the larger 
group $G_c$.  In particular 
we obtain an automatic extension of a large class of 
irreducible representations of $H$ to irreducible 
semibounded representations of $G_c$. 
For irreducible domains, the irreducible semibounded representations 
of $G_c$ are classified in \cite{Ne10c}, and this in turn leads to a 
classification of the corresponding representations of $H$. 

One of the central motivations to study results like the 
L\"uscher--Mack Theorem is  that there are natural sources 
of involutive representations of real Olshanski semigroups 
$S_H(W)$. Here the constructions based on ``reflection 
positivity'' are of particular interest because 
of their connections with euclidean, resp., relativistic 
quantum field theories (cf.\ \cite{JOl00}, \cite{NO11}, 
\cite{LM75}, \cite{GJ81}). In this context the semigroup 
$S_H(W)$ constitutes a bridge between unitary representations 
of the groups $G$ and $G_c$. However, these semigroups 
do not exist in all situations, where the passage from 
$G$ to $G_c$ is of interest, a typical example is 
the euclidean motion group $G = \R^4 \rtimes \SO_4(\R)$ 
and the Poincar\'e group $G_c = \R^4 \rtimes \SO_{1,3}(\R)$. 
This was the motivation for Fr\"ohlich, Osterwalder and Seiler 
to introduce the concept of a virtual representations of a 
symmetric Lie group (\cite{FOS83}). It would be very interesting 
to see if a suitable variant of 
this concept can be developed for Banach--Lie groups. 

\subsection*{Notation and conventions} 

If $\cH$ is a Hilbert space, we write $B(\cH)$ for the algebra of
bounded linear operators on $\cH$, $\GL(\cH)$ for its group of 
units, and $\U(\cH) \subeq \GL(\cH)$ for the unitary group. 

If $G$ is a topological group and $\cH$ a complex Hilbert space, 
then a {\it unitary representation of $G$ on $\cH$}, denoted 
$(\pi, \cH)$, is a homomorphism $\pi \: G \to \U(\cH)$ which is 
continuous with respect to the strong operator topology, i.e., all 
orbit maps $\pi^v \: G \to \cH, g \mapsto \pi(g)v$ are continuous. 

If $G$ is a Banach--Lie group, then we write $\g$ for its 
Lie algebra. It is a Banach--Lie algebra, i.e., a Banach space 
with a continuous Lie bracket. 

If $S$ is a semigroup, we write $\lambda_s(t) = st$ and 
$\rho_s(t) = ts$ for left and right multiplications on $S$. 
If a neutral element in $S$ exists, it is denoted by~$e$. 

Let $M$ be a set. A kernel function $K \: M \times M \to \C$ 
is called {\it positive definite} if for 
$x_1,\ldots, x_n \in M$, $n \in \bN$, the matrix 
$(K(x_i,x_j))_{i,j=1,\ldots n}$ is positive definite. 
For any such kernel there exists a unique Hilbert subspace 
$\cH_K \subeq \C^M$ of functions on $M$ with continuous evaluation maps 
given by $f(m) = \lan f, K_m\ran$, $K_m(x) = K(x,m), x, m \in M$. 
In particular, the subspace $\cH_K^0$, spanned by the functions 
$K_m$, $m \in M$, is dense in $\cH_K$. 
If $\gamma \: M \to \cH$ is a function with values in a Hilbert space 
whose range is total in the sense that it spans a dense subspace and 
$\lan\gamma(x), \gamma(y) \ran = K(y,x)$ for $x,y \in M$, then we have a 
unitary map $\Phi_\gamma \: \cH \to \cH_K, \Phi_\gamma(v)(x) = \lan v,\gamma(x)\ran$ (cf.\ \cite[Thm.~I.1.6]{Ne00}).

\section{Geometric symmetric operators on reproducing kernel spaces} 

Let $\M$ be a Banach manifold,  $K\in C^{\infty}(\M\times\M,\bC)$ 
be a smooth positive definite kernel on $\M$ and 
$\cH_K \subeq \C^M$ be the corresponding reproducing kernel Hilbert space. 
According to \cite[Thm.~7.1]{Ne10b}, the map 
$\M\ra\cHk,m\mt K_m,$ is smooth, so that $\phi(m) = \lan \phi, K_m\ran$ 
for $\phi \in \cH_K$ implies that $\cH_K \subeq C^\infty(\M,\bC)$. 

Let $V$ be a vector field on $\M$, $\cL_V$ the associated derivation of $C^\infty(\M,\bC)$, and
$\phi_t^V$ be its local flow, defined at $m\in\M$ for
$0\leq t<\epsilon(m)$. 
The Lie derivative on functions is given by 
\begin{equation}
  \label{eq:liedet} 
(\cL_V\ph)(m):=\df{t}\ph(\phi_t^Vm).
\end{equation}
We also consider $\cL_V$ as an unbounded operator on  
$\cH_K$ defined on the domain
$$\D_V:=\{\ph\in\cH_K\mid\cL_V\ph\in\cH_K\}.$$

\begin{de}
We say that $V$ is {\it symmetric with respect to $K$} 
(or {\it $K$-symmetric}), if 
\[ (\cL_VK_m)(n)=\ov{(\cL_VK_n)(m)} \quad \mbox{for } \quad 
m,n \in \M.\] 
\end{de}

The key point in the preceding definition is that it can be expressed 
in terms of the kernel and the vector field, but the following proposition 
shows that we can draw interesting conclusions for the corresponding 
unbounded operator on $\cH_K$. 

\begin{pro} \label{prop:2.2} 
Let $V$ be a $K$-symmetric vector field. Then, for every $m\in\M$,
\begin{equation}\label{E:diff}
\cL_VK_{m}=\df{t}K_{\phi_t^Vm}\in\cH.
\end{equation}
In particular $\cH_K^0\subseteq\D_V$ and 
the operator ${\cL_V}|_{\cH_K^0}$ is symmetric.
If $\ph\in\cH$ then 
\begin{equation}\label{E:diff2}
\ps{\ph,\cL_VK_m}=(\cL_V\ph)(m) \quad \mbox{ for } \quad m \in \M.
\end{equation}
\end{pro}

\begin{proof}
Since $\M\ra\cH_K$,
$m\ra K_m$ is smooth, the derivative $\dif{t}K_{\phi_t^Vm}$
exists in $\cH_K$ and, for every $n\in\M$,
\begin{align*}
(\cL_VK_m)(n)
= \oline{(\cL_VK_n)(m)}
=\df{t}K(n,\phi_t^Vm)
=\Big(\df{t}K_{\phi_t^Vm}\Big)(n). 
\end{align*}
For every $\ph\in\cH$ we therefore have
\[\ps{\ph,\cL_VK_m}=\ps{\ph,\df{t}K_{\phi_t^Vm}}=\df{t}\ph(\phi_t^Vm)=
(\cL_V\ph)(m). 
\qedhere\] 
\end{proof}

\begin{rem} Note that \eqref{E:diff2} implies 
in particular 
\begin{equation}\label{E:symm}
\ps{\ph,\cL_VK_m}=\ps{\cL_V\ph,K_m} 
\quad \mbox{ for } \quad \ph\in\D_V. 
\end{equation}
\end{rem}

Below we recall Fr\"ohlich's Theorem on unbounded symmetric semigroups
as it is stated in \cite[Cor.~1.2]{FRO}. 
Actually Fr\"ohlich assumes that the Hilbert space $\cH$ is separable, 
but this is not necessary for the conclusion. Replacing the assumption 
of weak measurability by weak continuity,  
all arguments in \cite{FRO} work for non-separable spaces as well. 

\begin{theo}[Fr\"ohlich] \mlabel{thm:2.4} 
Let $H$ be a symmetric operator defined on the domain $\D$ dense in
the Hilbert space $\cH$. Suppose that for every $\Phi\in\D$ there
exists $\epsilon(\Phi)>0$ such that the equation
$$\dif{t}\Phi(t)=H\Phi(t)$$
has a solution satisfying $\lim_{t\ra0}\Phi(t)=\Phi$ and $\Phi(t)\in\D$
for $0\leq t<\epsilon(\Phi)$. Then the operator $H$ is essentially 
self-adjoint and $\Phi(t) = e^{t\oline H}\Phi$.
\end{theo}

We apply Fr\"ohlich's Theorem 
to our geometric operator $\cL_V$ from Proposition~\ref{prop:2.2}. The new 
feature is that we can even describe the closure of the essentially self-adjoint operator we obtain.

\begin{theo}[Geometric Fr\"ohlich Theorem]\mlabel{T:Froelichgeo}
Let $\M$ be a Banach manifold and $K$ be a smooth positive definite 
kernel. Then $\cH_K$ consists of smooth functions, and if 
$V$ is a $K$-symmetric vector
field on $\M$, then the Lie derivative $\cL_V$ 
defines an essentially self-adjoint 
operator $\cH_K^0 \to \cH_K$ whose closure 
$\cL_V^K$ coincides with $\cL_V|_{\D_V}$. 
Moreover, if the local flow $\Phi^V_t(m)$ of $m \in \M$ at time $t$ is defined, 
then 
\[ e^{t\cL_V^K} K_m = K_{\Phi^V_t(m)}.\] 
\end{theo}

\begin{proof}
We apply Fr\"ohlich's Theorem with 
$\D=\cH_K^0$ and $H={\cL_V}|_{\D}$. For
$\Phi=\sum_{j=1}^k\ha_jK_{m_j}$, $\ha_j\in\bC$, $m_j\in M$, we define
$$\Phi(t)=\sum_{j=1}^k\ha_jK_{\phi_t^V m_j}, 
\quad \mbox{ for } \quad t<\min(\epsilon(m_1),\dots,\epsilon(m_k)).$$
Then \eqref{E:diff} implies $\dif{t}\Phi(t)=\cL_V\Phi(t)$, so the assumptions of Fr\"ohlich's Theorem are satisfied, and $\cL_V|_{\cH_K^0}$ is essentially self-adjoint. Let us denote by $T_V$ its closure. If $\ph$
is an element of its domain $\D(T_V)$, then 
\eqref{E:diff2} leads for every $m\in \M$ to 
\begin{align*}
(T_V\ph)(m)
&=\ps{T_V\ph,K_m}=\ps{\ph,\cL_VK_m}=(\cL_V\ph)(m).
\end{align*}
We conclude that $\D(T_V)\subseteq\D_V$ and $T_V=\cL_V|_{\D(T_V)}$.
The other inclusion 
follows from \eqref{E:symm} which implies 
that $\D_V\subseteq \D((\cL_V|_{\cH_K^0})^*)$. 
Indeed, since $\cL_V|_{\cH_K^0}$ is essentially self-adjoint we have
$(\cL_V|_{\cH_K^0})^*=T_V$.
\end{proof}

\begin{rem} Theorem~\ref{thm:2.4} 
above extends directly to the case where $\M$
is a manifold modeled on a locally convex space, 
provided the vector field $V$ is
assumed to have a local flow. Indeed \cite[Thm.~7.1]{Ne10b}
is stated in this generality. 
\end{rem}

\section{The L\"uscher--Mack Theorem}\label{sec:LM}

Let $(G,\theta)$ be a symmetric Banach--Lie group. 
We also write $\theta$ 
for the corresponding automorphism of its Lie algebra $\g$, which 
leads to the {\it symmetric Banach--Lie algebra} $(\fg,\theta)$.
We write 
\[ \fg=\fh\oplus \fq \quad \mbox{ with } 
\quad \fh = \ker(\theta - \1) \quad \mbox{ and } \quad 
\fq = \ker(\theta + \1),\]
for the eigenspace decomposition of $\fg$ under $\theta$. 
Let $\eset\not=W \subeq \fq$ be an open convex cone invariant under 
$e^{\ad \fh}$ and $S = S_H(W)$ be a corresponding 
Olshanski semigroup in the sense of Definition~\ref{def:olshsem}. 
In particular, $H$ is a connected Lie group with Lie algebra $\fh$, 
but we do not assume that $H$ is contained in $G$. We write 
\[ \Exp \: W \to S = S_H(W)\] 
for the corresponding exponential map on~$W$. 
In polar coordinates the involution on  $S$ is given by 
\begin{equation}
  \label{eq:invol}
(h \Exp x)^* = (\Exp x) h^{-1} = h^{-1}\Exp(\Ad(h)x)
\end{equation}
(Remark~\ref{rem:a.5}). 

\begin{exs} (a) Let $\cH$ be a real Hilbert space and 
$G = \GL(\cH)_0$. We write $A^\top$ for the adjoint of an element 
$A \in B(\cH)$. Then $\theta(g) = (g^\top)^{-1}$ defines an involution 
on $G$ with $H = G^\theta = \OO(\cH)_0$ (the orthogonal group of $\cH$) and 
$\fq = \Sym(\cH)$ is the space of symmetric operators. 
In this case 
$G = \OO(\cH)_0 \exp(\Sym(\cH))$ actually is an Olshanski (semi-)group 
for $W = \fq= \Sym(\cH)$. 
Similarly statements hold for complex and quaternionic 
Hilbert spaces. 

(b) If $H$ is a connected Lie group for which 
$\Ad(H)$ leaves a compatible norm on $\fh$ invariant, 
then $\fh$ is called an {\it elliptic Banach--Lie algebra}. 
Then $H$ has a universal complexification $\eta \: H \to H_\C$ 
with a polar decomposition $H_\C = H \exp(i\fh)$ 
(\cite[Ex.~6.9]{Ne02}). A finite dimensional Lie algebra 
$\fh$ is elliptic if and only if it is compact, but 
the class of elliptic Lie algebras is quite large. In particular, it 
contains 
the algebra $\fu(\cA)$ of skew-hermitian elements of a $C^*$-algebra
$\cA$ and in particular the Lie algebra $\fu(\cH)$ of the full unitary
group $\U(\cH)$ of a complex Hilbert space $\cH$. 

(c) If $V$ is a Banach space and $W \subeq V$ an open convex cone, then 
$S = S_V(W) = V + i W \subeq V_\C$ is a complex Olshanski semigroup 
 with respect to the involution $(x + iy)^* := -x + iy$.  

(d) If $\cA$ is a unital $C^*$-algebra, 
$G := \cA^\times_0$ (the identity component of its group of units) 
and $S := \{ s \in G \: \|s\| < 1\}$, then $S$ is a complex 
Olshanski semigroup $S= S_{\U(\cA)_0}(W)$, where 
\[ W = \{ x \in \cA \: x^* = x, \sup(\Spec(x)) < 0\}.\] 
An important example is the semigroup of invertible strict contractions 
of a complex Hilbert space $\cH$. 

(e) Let $\cA$ be a unital $C^*$-algebra and $\tau = \tau^2  = 
\tau^* \in \cA$. For $a,b \in \cA$, 
we write $a < b$ if there exists an invertible element $c \in \cA$ with $b - a = c^*c$. Then 
$$ S:= \{ s \in \cA^\times \: s^*\tau s < \tau \} $$
is an open subsemigroup of $\cA$ with respect to multiplication. 
To see that it is non-empty, we observe that we may write 
$\tau = \1 - 2 p = (\1 - p) - p$ 
for a projection $p = p^* = p^2 \in \cA$. For $\lambda \in \C^\times$ and 
$s := \lambda (\1-p) + \lambda^{-1}p$ we then have 
$$ s^* \tau s 
= |\lambda|^2 (\1 - p) - |\lambda^{-1}|^2 p < \tau = (\1-p) - p $$
if and only if $|\lambda| < 1$. The semigroup $S$ is of the 
form $S_H(W)$ for 
\[ H = \{ g \in \cA^\times \: g^*\tau g = \tau \}.\] 

(f) As already mentioned in the introduction, the 
compression semigroups of (real forms of) symmetric (Hilbert) domains 
are also Olshanski semigroups (cf.\ \cite{Ne01}).
\end{exs}

\begin{de} Let $(S,*)$ be an {\it involutive Banach semigroup}, i.e., 
an involutive semigroup carrying a Banach manifold structure such that 
multiplication and inversion are smooth maps. 

(a) A homomorphism $\rho \: S \to B(\cH)$ is called a 
{\it $*$-representation} if $\rho(s^*) = \rho(s)^*$ for every $s\in S$. 
Such a representation is said to be {\it non-degenerate} if 
$\rho(S)\cH$ spans a dense subspace of $\cH$, which is equivalent to the 
condition that $\rho(S)v=\{0\}$ implies $v = 0$. 

(b) For a representation $(\rho, \cH)$ of $S$, a vector $v \in \cH$ 
is called {\it smooth} if its orbit map 
$\rho^v \: S \to \cH, s \mapsto \rho(s)v$ is smooth. 
We write $\cH^\infty$ for the subspace of {\it smooth vectors} 
and say that $(\rho, \cH)$ is {\it smooth} if $\cH^\infty$ is dense 
in $\cH$. 

(c) A $*$-representation $(\rho, \cH)$ of $S$ is called {\it locally 
bounded} if every $s \in S$ has a neighborhood on which 
$\|\rho(\cdot)\|$ is bounded. 

(d) For a unitary representation $(\pi,\cH)$ of a Lie group $G$, we 
also write $\cH^\infty := \cH^\infty(\pi)$ for the subspace of 
smooth vectors. This subspace 
carries the {\it derived representation} 
\[ \dd\pi \: \g \to \End(\cH^\infty), \quad 
\dd\pi(x)v := \derat0 \pi(\exp tx)v \] 
of the Lie algebra $\g$ of $G$. 
If $\pi$ is smooth, then these operators are essentially 
skew-adjoint and, for $x \in \g$, the closure $\oline{\dd\pi}(x)$ 
is the infinitesimal generator of the unitary one-parameter group 
$\pi(\exp tx)$. 
\end{de}

\begin{rem} \mlabel{rem:locbo} 
(a) Since the arguments in \cite[Prop.~5.1, Lemma 5.2]{Ne10b} 
apply also to semigroup actions, the first countability 
of $S$ implies that 
any strongly continuous representation $\rho \: S \to B(\cH)$ 
defines a continuous action $S \times \cH \to \cH$. The continuity 
in the points $(s,0)$ implies that $\rho$ is locally bounded. 

(b) If, conversely, $\rho$ is locally bounded and the 
orbit maps $\rho^v \: S \to \cH$ are continuous for 
every $v$ in a dense subspace, then $\rho$ is strongly continuous 
(\cite[Lemma~IV.1.3]{Ne00}). 
\end{rem}

In the following we are interested in 
non-degenerate strongly continuous 
$*$-representations $\rho:S = S_H(W) \ra B(\cH)$ on a 
Hilbert space $\cH$. 
First we observe that, although $H$ need not be contained in $S$, 
any non-degenerate $*$-representation of $S$ defines in a natural 
fashion a unitary representation of $H$.

\begin{pro}\mlabel{P:repH} 
For every non-degenerate 
$*$-representation $(\rho, \cH)$ of 
$S$ there exists a unique, not necessarily continuous, 
unitary representation $\rho_H \: H \to \U(\cH)$ 
satisfying 
\begin{equation}\label{E:repH}
\rho(hs)=\rho_H(h)\rho(s)\quad \mbox{ for } \quad h\in H,  s\in S.
\end{equation}
Its space of continuous vectors contains $\rho(S)\cH$ 
and its space of smooth vectors contains $\rho(S)\cHy$. 
In particular, $(\rho_H, \cH)$ is strongly continuous, resp., smooth if 
$(\rho, \cH)$ is. 
\end{pro}

\begin{proof} The existence of a unique homomorphism 
$\rho_H \: H \to \U(\cH)$ satisfying \eqref{E:repH} 
follows from the fact 
that $H$ acts on $S$ by unitary multipliers 
(Remark~\ref{rem:a.5}, \cite[Rem.~III.1.5]{Ne00}). 
The smoothness of the action of $H$ on $S$ now implies 
that $\rho(s)v$ has a continuous (smooth) orbit map under $H$  
if $v$ has a continuous (smooth) orbit map under $S$. 
This completes the proof. 
\end{proof}

 To apply the Hille--Yosida Theorem to the symmetric one-parameter semigroups
$\rho_x(t) := \rho(\Exp tx)$ for $x \in W$, we need to know that  
\begin{equation}\label{E:stcont}
\lim_{t\ra0}\rho_x(t)v=v\ \text{for}\ v\in\cH.
\end{equation}
As the following lemma shows, this follows from the non-degeneracy 
of the representation and strong continuity. 

\begin{lem} \mlabel{lem:3.4}
If $\rho$ is a non-degenerate $*$-representation of $S = S_H(W)$ and 
$x \in W$, then $\rho(\Exp x)\cH$ is dense in $\cH$. Furthermore, 
\eqref{E:stcont} holds if $\rho_x$ is strongly continuous on $\bR_{>0}$.
\end{lem}

\begin{proof} Let $v\in (\rho(\Exp x)\cH)^\bot$. 
In view of 
\[ \ps{\rho(\Exp x)w,v}=\ps{w,\rho(\Exp x)v},\]  
this is equivalent to $\rho(\Exp x)v=0$. 
Now \cite[Cor.~II.4.15]{Ne00} 
implies that 
\[ \rho\Big(\exp \frac{x}{n}\Big)v=0 \quad \mbox{ for every } \quad n > 0\] 
 and hence that $\rho(\Exp tx)v =0$ for every $t > 0$. 

Any $s\in S$ can be written $s=s_0\exp(tx)$
with $s_0\in S$ and some  $t>0$ 
because the left invariant vector field $V_x$ generates a local 
flow on $S$ (Remark~\ref{rem:a.9}). It follows that 
$\rho(s)v = \rho(s_0) \rho(\Exp tx)v = 0$. 
Since $\rho$ is a non-degenerate representation of $S$, it follows 
that $v = 0$, and hence that $\rho(\Exp x)\cH$ is dense in $\cH$. 

We see in particular, that the representation 
$\rho_x$ of $\bR_{> 0}$ is non-degenerate. 
Assume that $\rho_x$ is strongly continuous, hence locally 
bounded (Remark~\ref{rem:locbo}). Then 
Lemma~\cite[Lemma~VI.2.2]{Ne00} implies that 
$\rho_x$ extends uniquely to a strongly continuous representation 
on $\bR_{\geq 0}$, which implies \eqref{E:stcont}. 
\end{proof}

\begin{de} Let $(\rho, \cH)$ be a non-degenerate 
strongly continuous $*$-rep\-re\-sen\-ta\-tion of $S$. 
In view of Lemma~\ref{lem:3.4}, we obtain for 
every $x \in W$ a self-adjoint operator 
$$\ov{\dro}(x)\xi:=\df{t}\rho(\Exp{tx})\xi, $$
the generator of the symmetric one-parameter semigroup $\rho_x$. 
It is defined on the subspace $\D(\odro(x))$ where the derivative exists 
(Hille--Yosida Theorem). For $x \in \fh$, we likewise write 
$\odro(x)$ for the generator of the 
corresponding strongly continuous 
unitary one-parameter group $\rho_x(t) := \rho_H(\exp tx)$ 
(Proposition~\ref{P:repH}; Stone's Theorem). 
\end{de}

The following theorem is our main result. 

\begin{theo} \mlabel{thm:3.6} 
Let $G_c$ be a simply connected Lie group with Lie algebra 
$$\fgc=\fh\oplus i\fq \subeq \g_\bC, $$
 and $\rho:S = S_H(W) \ra B(\cH)$ be a non-degenerate strongly continuous smooth $*$-representation. Then there exists a unique 
smooth unitary
representation $(\pi, \cH)$ of $G_c$ on $\cH$ whose space of smooth vectors is
contained in $\D(\odro(x))$ for every $x\in\fh\cup W$, and whose derived
representation satisfies
\begin{equation}\label{E:inflm}
\dpi(x+iy)\subeq 
\odro(x)+i\,\odro(y)\ \text{for}\ x\in\fh\ \text{and}\ y\in W.
\end{equation}
\end{theo}

We will see (Remark~\ref{R:dom}) that the space $\cHy(\pi)$ 
of smooth vectors for the representation $\pi$ of $G_c$ 
coincides with
\begin{equation*}
\bigcap_{x_j\in\fh\cup W,\ n\in\bN}\D\big(\odro(x_n))\dots
\odro(x_1)\big).
\end{equation*} 
Accordingly, our strategy is to define a representation 
of the Lie algebra $\g_c$ on this space so that \eqref{E:inflm} is 
satisfies and then verify that 
we can use the results in \cite{Mer10} to show that 
this representation of $\g_c$ integrates to a representation of $G_c$.

\begin{rem} (a) Let $q_H \: \tilde H \to H$ be the simply connected
 covering group of $H$. Then we have a unique morphism 
$\iota \: \tilde H \to G_c$ integrating the inclusion 
map $\fh \to \g_c$. The relation $\pi \circ \iota = \rho_H \circ q_H$ 
now implies that 
$\iota(\ker q_H) \subeq \ker \pi$. 
As $\ker q_H$ acts trivially on $\g$, it also acts trivially on $\g_c$, 
i.e., $\iota(\ker q_H) \subeq Z(G_c)$. In particular, it is a normal 
subgroup and $\pi$ actually factors through a representation 
of the quotient $G_c/\iota(\ker q_H)$. 
If $\iota(\ker q_H)$ is discrete, this quotient is a Lie group 
with the same Lie algebra. 

Here is an example showing that the group 
$\iota(\ker q_H)$ need not be discrete. 
We consider the Lie algebra $\g = \fsl_2(\R) \oplus \R$ 
with the involution $\theta(x,t) = (-x^\top,t)$, which leads to 
$\g_c = \su_2(\C) \oplus \R$, and 
$G_c := \SU_2(\C) \times \R$ is a corresponding simply connected 
group. 

For the Lie algebra $\fh = \g^\theta = \so_2(\R) \oplus \R$, 
the corresponding simply connected group is 
$\tilde H = \R^2$, and the kernel for the adjoint action of 
$\tilde H$ on $\g$ is isomorphic to $\Z \times \R$ 
with $\ker \iota = 2\Z \times \{0\}$. Now 
$\Gamma := \Z (2,1) \oplus \Z(0,\sqrt 2)$ is a discrete subgroup 
of $\tilde H$ acting trivially on $\g$, so that 
$H := \tilde H/\Gamma$ satisfies our condition imposed for the 
construction of Olshanski semigroups and 
$\ker q_H = \Gamma$. Now $\iota(\Gamma) = \{\1\} \times 
(\Z + \sqrt 2\Z)$ is not discrete in $G_c$. 

(b) Let $q_S \: \tilde S = S_{\tilde H}(W)$ be the universal covering 
of the Olshanski semigroup $S = S_H(W)$ and 
$(\rho,\cH)$ be a smooth locally bounded $*$-representation of 
$S$, so that Theorem~\ref{thm:3.6} leads to a unitary representation 
$(\pi,\cH)$ of $G_c$. Clearly, the representation $\tilde\rho :=\rho \circ q_S$ 
of $\tilde S$ leads to the same representation of $G_c$. 

Now $\pi \circ \iota = \tilde\rho_H$ implies that 
$\ker(\tilde\rho_H) \supseteq \ker \iota$, so that the representation 
$\tilde\rho$ of $\tilde S$ actually factors through a 
representation of the semigroup 
$S_c := S_{H_c}(W)$, where $H_c = \iota(\tilde H)$. 
From the point of view of the representation theory of 
the group $G_c$, all representations obtained by 
Theorem~\ref{thm:3.6} can also be obtained from the semigroup 
$S_c$, which is locally isomorphic to $S$. 
\end{rem}

The representation $(\pi,\cH)$ of $G_c$ obtained from
the L\"uscher--Mack Theorem has the remarkable property that 
the spectrum of the operator $i \odpi(x)$ is bounded 
from below for every $x$ in the cone $iW \subeq \g_c$. 
In Section~\ref{S:holext} we will 
prove a converse to Theorem~\ref{thm:3.6}: A $-iW$-semibounded 
representation $\pi$ of $G_c$ comes from 
a strongly continuous smooth $*$-representation of 
$S_H(W)$, where $H = \langle \exp_{G_c} \fh \rangle$ is 
the identity component of the 
group of fixed points for the involution 
$\theta_c$ on $G_c$ acting on the Lie algebra by 
$x + iy \mapsto x -iy$ for $x \in \fh, y \in \fq$.

\begin{rem} It is instructive to take a closer look at the special case 
$\g = \fh_\C$ with $\fq = i\fh$, i.e., where $\theta(z) = \oline z$ 
is complex conjugation with respect to the real form $\fh$. 
Then $S_H(W)$ is a complex Olshanski 
semigroup (cf.\ Definition~\ref{def:olshsem}(b)). 
Then the complexification of 
$\g$ can be realized by the embedding 
\[ \eta \: \g \to \g \oplus \g, \quad z \mapsto (z,\oline z),\] 
which leads to $\g_\C \cong \g \oplus \g$, a direct sum of two complex 
Lie algebras. In this picture the 
Lie algebra $\g_c = \fh + i\fq$ is given by 
\[ \{ (x,x) \: x \in \fh\} +  i\{ (y,-y) \: y \in i\fh\}   
= \{ (x,x) \: x \in \fh\} +  \{ (y,-y) \: y \in \fh\}   
= \fh \oplus \fh.\] 
We conclude that $G_c \cong \tilde H \times \tilde H$, 
where $\tilde H$ is the simply connected covering group of~$H$. 

If $\alpha \: \g \to \End(\cD)$ is a complex linear representation 
of $\g$ on the complex linear space $\cD$, then the complex linear 
extension to $\g_\C \cong \g \oplus \g$ is given by 
\[ \alpha_\C(z,w) = \shalf\big(\alpha(z + \oline w) 
+ i \alpha(i\oline w - iz)\big) = \alpha(z) \] 
because 
\[ (z,w) 
= \shalf \big((z+\oline w,\oline z + w) 
+ i (i\oline w-iz, \oline{i\oline w - iz})\big).\] 
Therefore the corresponding unitary representation 
$\pi_c$ of $G_c$ is given by 
$\pi_c(h_1, h_2) = \rho_H(h_1)$. 
\end{rem}

\section{Proof of the L\"uscher--Mack Theorem} \mlabel{sec:4}

The following lemma permits us to reduce the proof of Theorem~\ref{thm:3.6} 
to the case where $(\rho, \cH)$ is cyclic and generated by a smooth 
vector. 

\begin{lem} \mlabel{lem:3.7} 
A non-degenerate strongly continuous smooth 
$*$-representation is a direct sum
of cyclic representations with smooth cyclic vectors.
\end{lem}

\begin{proof}
The set of all families $(\cH_j)_{j\in J}$ of mutually orthogonal
closed $S$-invariant subspaces which contain a smooth cyclic vector is well ordered by inclusion and by Zorn's Lemma it has a 
maximal element $(\cH_j)_{j\in J}$. Let $\K=\ov{\bigoplus_{j\in J}\cH_j}$. 
Then $\K$ and $\K^\perp$ are $S$-invariant, and we claim that if $\K^\perp$ is
non-zero, it contains a smooth vector. Indeed, if $\K\neq\cH$, there exists a smooth vector $w$
which is not in $\K$. Let us write $\mathrm{pr}$ for the orthogonal projection on 
$\K^\perp$ and $v:=\mathrm{pr}(w)$. Then the relation 
$\rho(s)v=\mathrm{pr}(\rho(s)w)$, $s\in S$,
shows that $v$ is a smooth vector in $\K^\perp$. 
Since $\rho$ is non-degenerate, 
we have $v\in\ov{\rho(S)v}$ 
(cf.\ \cite[Lemma~II.2.4]{Ne00}), hence $\ov{\Spann{\rho(S)v}}$ is a  closed 
$S$-invariant subspace with a smooth cyclic vector orthogonal to each 
$\cH_j$, $j\in J$. But this contradicts the maximality of $(\cH_j)_{j\in J}$. Therefore 
$\cH=\K$, and as representations of $S$, we have
$\cH\simeq\widehat{\bigoplus}_{j\in J}\cH_j$.
\end{proof}

From now on we assume that $v_0 \in \cH^\infty$ is a cyclic vector. 
Then we obtain a smooth positive definite kernel  
$$K(s_1,s_2):=K_{s_2}(s_1):=\ps{\rho(s_1s_2^*)v_0,v_0}$$
on $S$ which leads to a unitary map 
\[ \Psi \: \cH \to \cH_K, \quad \Psi(v)(s) 
:= \rho^{v,v_0}(s) 
= \lan \rho(s)v,v_0\ran = \lan v, \rho(s^*)v_0\ran \] 
onto the reproducing kernel space $\cH_K\subseteq C^\infty(S,\bC)$. 
Indeed, $\Psi$ 
intertwines $\rho$ with the representation 
of $S$ on $\cH_K\subseteq C^{\infty}(S,\bC)$ 
given by $(\rho(s_2)\ph)(s_1)=\ph(s_1s_2)$. 
Therefore it suffices to prove Theorem~\ref{thm:3.6} only for 
the representation $(\rho, \cH_K)$ of $S = S_H(W)$.

The function $\ph_0=\rho^{v_0,v_0} = \Psi(v_0)$ satisfies 
$\rho(s)\ph_0=K_{s^*}$, so that 
\begin{equation} \label{eq:8}
 \cHr=\spann(\rho(S)\ph_0) \subeq \cH_K^\infty 
\quad \mbox{ and } \quad 
\rho(s_2)K_{s_1}=K_{s_1s_2^*} 
\end{equation}
for $s_1 \in S, s_2 \in S \cup H$.
To each $x\in\fg$ we associate the corresponding left invariant 
vector field $V_x$ on $S$ (cf.\ Definition~\ref{D:tangiso}) 
and write the corresponding integral curves as $s\exp(tx)$ 
(cf.\ Remark~\ref{rem:a.9}). 
As in \eqref{eq:liedet}, we set 
$$\cL_x:=\cL_{V_x}\quad \mbox{ and } \quad 
\cL_x^K := \cL_x\res_{\cD_x} \quad \text{for}\ \D_x:=\D_{V_x} 
= \{ \phi \in \cH_K \: \cL_x\phi \in \cH_K\},$$
and extend this definition $\bC$-linearly to any $x\in\fg_\bC$. Then we have 
on $C^\infty(S,\bC)$ the relation 
\begin{equation}
  \label{eq:liederbrack}
[\cL_x,\cL_y]=\cL_{[x,y]}.
\end{equation}

We recall that a {\it core} of a self-adjoint operator 
$A \: \cD \to \cH$ is a dense subspace $\cD_0 \subeq \cD$ for which 
$A = \overline{A\vert_{\cD_0}}$. 

\begin{pro} \mlabel{prop:Lieder}
Let $x\in \fh\cup W$. Then $\cHr\subseteq\D(\odro(x))$, 
$\cHr$ is a core for $\odro(x)$, and 
\begin{equation}
  \label{eq:lieder-id}
\ov{\dro}(x)=\cL_x^K. 
\end{equation}
\end{pro}

\begin{proof}  
For $x\in \fh$, the first assertion follows from 
\eqref{eq:8} and Proposition~\ref{P:repH}.
For $x\in W$ and $s\in S$, the map $t\mt 
\rho(\exp(tx)s)\ph_0$ extends to a smooth map on
some $0$-neighborhood, and for $t\geq 0$ it coincides with
$e^{t\odro(x)}\rho(s)\ph_0$.
Therefore $\rho(s)\ph_0$ is contained in the domain 
$\cD(\oline{\dd\rho}(x))$ of $\oline{\dd\rho}(x)$.

In the following we also write $\rho(h) := \rho_H(h)$ 
for $h \in H$.
For $x\in\fh\cup W$, the space $\cHr$ is invariant under 
$\rho(\exp tx)$, $t> 0$, hence is a core for $\odro(x)$ 
(\cite[Prop.~1.7]{ENGELNAGEL}), i.e., 
$\oline{\dd\rho}(x) = \oline{\oline{\dd\rho}(x)\res_{\cH_K^0}}$. 

Finally, 
let $x\in \fh\cup W$ and $\ph\in\D(\odro(x))$. Then
\begin{align*}
& \big(\oline{\dd\rho}(x)\ph\big)(s) 
= \ps{\df{t}\rho(\exp{tx})\ph,K_s}=\df{t}\ps{\ph,\rho((\exp{tx})^*)K_s}\\
&=\df{t}\ps{\ph,K_{s\exp{tx}}}=\df{t}\ph(s\exp{tx})=(\cL_x\ph)(s),
\end{align*}
shows that $\cL_x$ extends $\odro(x)$. Further, for $\psi \in \cD_x$, 
\begin{align*}
\lan \cL_x \psi, K_s \ran 
&= (\cL_x\psi)(s) = \derat0 \psi(s\exp(tx)) 
= \lan \psi, \derat0 K_{s\exp tx}\ran \\
&= \lan \psi, \derat0 \rho(\exp tx)^* K_{s}\ran 
= \lan \psi, \oline{\dd\rho}(-\theta(x))K_s\ran.
\end{align*}
shows that $(\odro(-\theta(x))|_{\cHr})^*=\odro(x)$ extends $\cL_x$. 
This proves \eqref{eq:lieder-id}. 
\end{proof}

The preceding proposition shows that, 
for any $x\in\fh\cup W$, $\cL_x$ leaves the subspace 
\begin{align*}
\D
&:=\bigcap_{n\in \bN,\ x_n,\dots,x_1\in\fh\cup W}
\D(\ov{\dro}(x_n)\dots\ov{\dro}(x_1)) \\
&=\{ \ph \in \cH_K \: (\forall n \in \bN)(\forall x_n,\dots,x_1\in\fh\cup W)
\ \cL_{x_1} \cdots \cL_{x_n}\ph \in \cH_K \} 
\end{align*}
invariant, and by linearity this is also true for $x\in\fgC$. 
Therefore we set
$$\ha(x):=\cL_x|_\D \quad \mbox{ for } \quad x\in\fgC$$  
and observe that 
\begin{equation}\label{E:stab}
\cD = \{ \ph \in \cH_K \: (\forall n \in \bN)
(\forall x_1,\dots,x_n \in \g)\ \cL_{x_1} \cdots \cL_{x_n}\ph \in \cH_K \}.
\end{equation}

\begin{lem}\mlabel{L:Lieder} If $x_1,\dots x_n\in\fg$, 
$n\in\bN$, then
$$\cL_{x_n}\dots\cL_{x_1}\cHr\subseteq\cH_K, $$
and for $s \in S$, 
$$\cL_{x_n}\dots\cL_{x_1}\rho(s)\ph_0=\frac{\partial t^{n}}{\partial t_n\dots
\partial t_1}\Bigr|_{t_n=\dots=t_1=0}\rho(\exp(t_{1}x_1)\dots\exp(t_{n}x_n)s)
\ph_0.$$
\end{lem}

\begin{proof} The map
$$(t_1,\dots,t_n,s)\mt \rho(\exp(t_{1}x_1)\dots\exp(t_{n}x_n)s)\ph_0,$$
defined for small enough $t_j$, is a smooth $\cH_K$-valued map 
(cf.\ Remark~\ref{rem:a.9}). Hence
\[ \frac{\partial t^{n}}{\partial t_n\dots\partial t_1}\Bigr|_{t_n=\dots=t_1=0}\rho(\exp(t_
{n}x_n)\dots\exp(t_{1}x_1)s)\ph_0\in\cH_K\]  and 
\begin{align*}
&\Big(\frac{\partial t^{n}}{\partial t_n\dots\partial t_1}\Bigr|
_{t_n=\dots=t_1=0}\rho(\exp(t_{n}x_n)\dots\exp(t_{1}x_1)s)\ph_0\Big)(s_0)\\
&=\frac{\partial t^{n}}{\partial t_n\dots\partial t_1}\Bigr|_{t_n=\dots=t_1=0}
(\rho(s)\ph_0)(s_0\exp(t_{n}x_n)\dots\exp(t_{1}x_1))\\
&=\big(\cL_{x_n}\dots\cL_{x_1}(\rho(s)\ph_0)\big)(s_0).\qedhere
\end{align*}
\end{proof}

\begin{pro} \mlabel{P:repgc} The domain $\D$ contains $\cHr$ 
and is dense in $\cH_K$.
The map $\ha:\fgc\ra \End(\D)$ is a strongly continuous representation
of $\fgc=\fh\oplus i\fq$ by skew-symmetric operators in the sense that 
all the maps $\g_c \to \cH_K, x \mapsto \alpha(x)v$ are continuous. 
More generally, for each 
$\ph \in \cD^1 := \bigcap_{x \in \g} \cD_x,$
the map $\g_\C \to \cH_K, x \mapsto \cL_x \ph$ is continuous. 
\end{pro}

\begin{proof} The first assertion follows from \eqref{E:stab} and
Lemma~\ref{L:Lieder}.
The map $\ha$ is a Lie algebra homomorphism because $x\mt\cL_x$ is so 
by \eqref{eq:liederbrack}.
For the strong continuity it suffices to show that, for 
$\ph \in \cD^1$,  
the graph of the map $\g_c \to \cH_K, x \mapsto \cL_x \ph$ is closed 
(cf.\ \cite[Lemma 4.2]{Ne10a}). This follows if, for 
each $s \in S$ the map 
\[ \fgc\ra\cH,\ x\mt (\cL_x\ph)(s)= d\phi(s)V_x(s) \] 
is continuous. That this is the case on the real Lie algebra 
$\g$ follows from the fact that $\phi$ is a smooth function 
on $S$ because $\g \to T_s(S), x \mapsto V_x(s)$ is a topological isomorphism 
(cf.\ Definition~\ref{D:tangiso}). Clearly, the continuity 
is inherited by the complex linear extension to $\g_\C$.
This completes the proof. 
\end{proof}

To show that the representation of $\fgc$ integrates to a 
continuous unitary representation of $G_c$ we will use the following theorem
\begin{theo}[\cite{Mer10}]\mlabel{T:intcrit}
Let $G_c$ be a simply
connected Banach--Lie group with Lie algebra $\g_c$. Assume 
that $\g_c=\fa_1\oplus\fa_2$ where $\fa_1$ and $\fa_2$ are closed subspaces. 
Let $\ha$ be a 
strongly continuous representation of $\g_c$ on a dense subspace  $\D$ 
of a Hilbert space $\cH$ such that 
for every $x\in\fa_1\cup\fa_2$, $\ha(x)$ is essentially skew-adjoint, 
$e^{\ov{\ha(x)}}\D\subseteq\D$ and 
$$e^{\ov{\ha(x)}}\ha(y)e^{-\ov{\ha(x)}}=\ha(e^{\ad x}y)
\quad \mbox{ for } \quad y\in\g_c.$$
Then $\ha$ integrates to a continuous unitary representation 
$(\pi, \cH)$ of $G_c$ for which $\cD \subeq \cH^\infty(\pi)$ is a 
$G_c$-invariant subspace and 
$\alpha(x) = \dd\pi(x)|_{\cD}$ for $x \in \g_c$. In particular, 
for $x \in \g_c$, the infinitesimal generator $\oline{\dd\pi}(x)$ 
of the one-parameter group $\pi(\exp tx)$ coincides with 
the closure $\oline{\alpha(x)}$. 
\end{theo}

The next two propositions ensure 
that the assumptions of the theorem are satisfied by the representation $\ha$ for 
$\fa_1 = \fh$ and $\fa_2 = i\fq$. The first one uses the Geometric
Fr\"ohlich Theorem~\ref{T:Froelichgeo}. 

\begin{pro}\mlabel{P:sa} For $x\in\fh\cup i\fq$, 
the operator $\ha(x)$ is essentially skew-adjoint with 
closure~$\cL_x^K$. 
\end{pro}

\begin{proof}
For $x\in\fh\cup iW$, we know that $\cHr$ is a core for the 
skew-adjoint operator $\odro(x)=\cL_x^K$ (Proposition~\ref{prop:Lieder}).
 Since $\cHr\subseteq\D\subseteq\D_x$, the larger subspace 
$\D$ is also a core for $\cL_x^K$. This proves the assertion 
for $x \in \fh \cup iW$. 

Now let $x\in\fq$ be a general element. Then the vector field
$V_x$ is $K$-symmetric, since for every $s_1, s_2\in S$, we have
$$K(s_1\exp tx,s_2)=\ph_0(s_1(\exp tx)s_2^*)=\ph_0(s_1(s_2\exp tx)^*)=K(s_1,
s_2\exp tx).$$
We can therefore apply Theorem~\ref{T:Froelichgeo} and 
Proposition~\ref{prop:Lieder} to see that $\cL_x^K$ is
a self-adjoint operator, and that $\cL_x\res_{\cH_K^0}$ 
is essentially self-adjoint. 
Writing $x=x_+-x_-$ with $x_\pm\in W$, we see that we also have 
$\cL_x = \cL_{x_+} - \cL_{x_-}$, so that 
\[ \cH_K^0 \subeq \D\subseteq\D_{x_+}\cap\D_{x_-}\subseteq\D_x,\] 
and thus $\ha(ix)$ is essentially skew-adjoint 
with closure $\cL_{ix}^K$.
\end{proof}

\begin{lem} \mlabel{L:comrel} Let $x\in\fh\cup W$, $y\in\fg$, 
and $\ph\in\cH_K$ such that 
$\phi \in \cD_{e^{-\ad x}y}$. Then $\rho(\exp x)\ph\in\D_y$ and
$$\cL_y(\rho(\exp x)\ph)=\rho(\exp x)\cL_{e^{-\ad x}y}\ph.$$
\end{lem}
\begin{proof}
For every $s\in S$, $\cL_{e^{-\ad x}y}\ph\in\cH$ leads to 
\begin{align*}
&\ \ \ \ \ps{\rho(\exp x)\cL_{e^{-\ad x}y}\ph,K_s}\\
&=\ps{\cL_{e^{-\ad x}y}\ph,
\rho((\exp x)^*)K_s}=\ps{\cL_{e^{-\ad x}y}\ph,K_{s\exp x}}\\
&=\df{t}\ph(s\exp x\exp(te^{-\ad x}y))=\df{t}\ph(s\exp ty\exp x)\\
&=\cL_{y}(\rho(\exp x)\ph)(s)=\ps{\cL_y(\rho(\exp x)\ph),K_s}.\qedhere
\end{align*}
\end{proof}

\begin{lem}\mlabel{L:ext}
Let $A:\D\ra\cH$ be a symmetric operator and $\D_1\subseteq\D$ a dense 
subspace. If 
$A_1:=A|_{\D_1}$ is essentially self-adjoint, then $A$ is also 
essentially self-adjoint, with $\ov{A}=\ov{A_1}$.
\end{lem}

\begin{proof} From $A_1\subseteq A \subeq A^*$ it follows that 
\[ \oline A \subeq A^* \subeq A_1^* = \oline{A_1} \subeq \oline A\] 
(cf.\ \cite[Thm.~VIII.1]{RESI}). 
Therefore $\oline A = A^* = \oline{A_1}$ is self-adjoint. 
\end{proof}

\begin{pro} \mlabel{prop:4.9}
For $x\in\fh\cup i\fq$,
$e^{\ov{\ha(x)}}\D\subseteq\D$. Moreover for any
$y\in\fgc$ we have 
\begin{equation}\label{E:comrel2}
e^{\ov{\ha(x)}}\ha(y)e^{-\ov{\ha(x)}}=\ha(e^{\ad x}y).
\end{equation}
\end{pro}

\begin{proof} In the following we also write $\rho(h) := \rho_H(h)$ 
for $h \in H$. \\ 
{\bf Step 1:} Let $x\in\fh\cup W$ and $\ph\in\D$. Then 
Lemma~\ref{L:comrel} implies by induction that for every $n\in\bN$ and 
$y_1,\dots,y_n\in\fh\cup W$,
\begin{equation*}\label{E:multcom}
\cL_{y_n}\dots\cL_{y_1}\rho(\exp x)\ph=\rho(\exp x)\cL_{e^{-\ad x}y_n}\dots\cL_{e^{-\ad 
x}y_1}\ph\in\cH,
\end{equation*} 
and hence $e^{\odro(x)}\ph=\rho(\exp x)\ph\in\D(\odro(y_n)\dots\odro(y_1))$.
It follows that  
$$e^{\odro(x)}\D\subseteq\D,$$
and that we have for $\ph\in\D$ and (by linearity) for $y\in\fgc$, 
\begin{equation}\label{E:comrel3}
\ha(y)e^{\odro(x)}\ph=e^{\odro(x)}\ha(e^{-\ad x}y)\ph.
\end{equation}
In particular \eqref{E:comrel2} holds for $x\in\fh$.
Now let $\psi\in\D$. Then \eqref{E:comrel3} can be written
\begin{equation}\label{E:comrel}
\ps{-e^{\odro(x)}\ph,\ha(y)\psi}=\ps{e^{\odro(x)}\ha(e^{-\ad x}y)\ph,\psi},
\end{equation}
and this last equation is all we need for the following.

{\bf Step 2:} Assume that $x\in W$. Then the spectrum of $\odro(x)$ 
is bounded from above and hence
$t\mt e^{t\odro(x)}$ extends to a strongly continuous holomorphic semigroup 
$$\bC^+=\{z\in\bC\mid \Re z\geq0\}\ra B(\cH),\ z\mt e^{z\odro(x)},$$
which is holomorphic on $\interior(\bC^+)$ (cf.~
\cite[Prop.~9.9]{HILGERTNEEB} or \cite[Prop.~VI.3.2]{Ne00}).
For $\ph\in\D^1$, the map $\fg_\C\ra\cH$, $x\mt\cL_x\phi$ is $\bC$-linear
and continuous (Proposition~\ref{P:repgc}), hence the function 
$\bC\ra\cH$, $z\mt \cL_{e^{-z\ad x}y}\ph$ is analytic. Since the map 
\[B(\cH)\times\cH\ra\cH, \quad (T,\xi)\mt T\xi \] 
 is $\bC$-bilinear and continuous,
it follows that the map 
\[ \interior(\bC^+)\ra\cH,\quad z\mt e^{z \odro(x)}
\cL_{e^{-z\ad x}y}\ph \] 
is analytic. 
By the Analytic Continuation Principle, 
the equality \eqref{E:comrel} implies 
$$\ps{-e^{z\odro(x)}\ph,\ha(y)\psi}=\ps{e^{z \odro(x)}\cL_{e^{-z\ad x}y}\ph,
\psi}\quad \text{for}\ z\in\interior\bC^+, \psi \in \cD.$$ 
We then have by continuity
\begin{equation}\label{E:icomrel}
\ps{-e^{\pm i \odro(x)}\ph,\ha(y)\psi}
=\ps{e^{\pm i\odro(x)}\cL_{e^{\mp i\ad x}y}\ph,\psi}.
\end{equation}
This shows that, for $y\in\fh\cup iW$, $e^{\pm i\odro(x)}\ph\in\D(\ha(y)^*)
=\D(\odro(y)) = \cD_y$. We thus arrive at 
\[ e^{\pm i\odro(x)}\cD^1 \subeq \cD^1 
\quad \mbox{ with } \quad 
\cL_y e^{\pm i \odro(x)}= e^{\pm i\odro(x)}\cL_{e^{\mp i\ad x}y} 
\quad \mbox{ on } \quad \cD^1.\] 
By induction, we now obtain 
\[ e^{\pm i\odro(x)}\cD = \cD
\quad \mbox{ with } \quad 
\alpha(y) e^{\pm i \odro(x)}\res_\cD 
= e^{\pm i\odro(x)}\alpha(e^{\mp i\ad x}y).\] 

{\bf Step 3:} For 
$n\in\bN$, $x_1,x_2\in W$, $y\in\fgc$ and $\ph,\psi\in\D^1$ we now obtain 
\begin{multline}\label{E:rec}
\ps{-\left(e^{i\odro(x_1)}e^{-i\odro(x_2)}\right)^n\ph,\odro(y)\psi}\\
=\ps{\cL_{\left(e^{i\ad x_2}e^{-i\ad x_1}\right)^ny}\ph,
\left(e^{i\odro(x_2)}e^{-i\odro(x_1)}\right)^n\psi}
\end{multline}
from 
\[ \cL_y \left(e^{i\odro(x_1)}e^{-i\odro(x_2)}\right)^n\ph 
= \left(e^{i\odro(x_1)}e^{-i\odro(x_2)}\right)^n
\cL_{\left(e^{i\ad x_2}e^{-i\ad x_1}\right)^ny}\ph.\] 

Now let $x\in\fq$ and write it as  $x=x_1-x_2$ with $x_1,x_2\in W$.
Since $\odro(x_1)-\odro(x_2)$ is essentially self-adjoint on
$\D$ (Proposition~\ref{P:sa}), it is essentially self-adjoint as an operator on its domain
$\D(\odro(x_1))\cap\D(\odro(-x_2))$, and its closure is $\ov{\ha(x_1-x_2)}$
(Lemma~\ref{L:ext}). We therefore have Trotter's 
Product Formula \cite[Thm.\ VIII.31]{RESI}:
\begin{equation}
\lim_{n\ra\infty}\left(e^{i\frac{\odro(x_1)}n}e^{-i\frac{\odro(x_2)}n}\right)^n f=e^{i\ov{\ha(x_1-x_2)}}f \quad \mbox{ for } \quad f\in\cH_K.
\end{equation}
Replacing $x_j$ by $\frac{x_j}{n}$, $j=1,2$, in \eqref{E:rec} and taking the limit
we obtain
\begin{equation}
\ps{-e^{i\ov{\ha(x_1-x_2)}}\ph,\cL_y \psi}=\ps{e^{i\ov{\ha(x_1-x_2)}}
\cL_{e^{i\ad(x_2-x_1)}y}\ph,\psi} 
\quad \mbox{ for } \quad \ph, \psi \in \cD^1. 
\end{equation}
As above we now obtain by induction that 
$e^{i\ov{\ha(x)}}\D\subseteq\D$ and that \eqref{E:comrel2}
holds for $x\in i\fq$.
\end{proof}

\begin{proof}[Proof of Theorem~\ref{thm:3.6}:]
In view of Proposition~\ref{prop:4.9}, we obtain from 
Theorem~\ref{T:intcrit} a unitary 
representation of $G_c$ whose space of smooth vectors contains $\D$ 
and such that $\dpi|_\cD=\ha$ and $\odpi(x)=\ov{\ha(x)}$ for $x\in\fgc$. 
We conclude with Proposition~\ref{prop:Lieder} and 
\ref{P:sa} that $\oline{\dd\pi}(x)=\odro(x)$ for $x\in\fh$ and 
$\oline{\dd\pi}(ix)=i\odro(x)$ for $x\in W$. 
\end{proof}
\begin{rem}\label{R:dom} Since $\cD\subseteq\cHy(\pi)$, it immediatly follows
from  
\[ \cHy(\pi)=\bigcap_{x_j\in\fgc,\ n\in\bN}\D(\oline{\dd\pi}(x_n)
\dots\oline{\dd\pi}(x_1))\subseteq\cD\] 
(\cite[Lemma 3.4, Remark 8.3]{Ne10b}), that
 $\D=\cHy(\pi)$.
\end{rem}

\section{Holomorphic extension of semibounded representations}\label{S:holext}

In this section we obtain a result which is a converse
to the L\"uscher--Mack Theorem. It is new even in the 
finite dimensional setting. At the same time, 
we obtain the existence of holomorphic extensions for 
semibounded unitary representations.

\begin{de} \mlabel{D:5.3} 
Let $G$ be a Banach--Lie group with Lie algebra $\fg$. 
For a smooth unitary representation $(\pi,\cH)$ of $G$ we 
consider the map
\[ s_\pi:\fg\ra\bR\cup\{\infty\},\quad 
s_\pi(x):=\sup\left(\Spec(i\odpi(x))\right).\] 

(a) A smooth unitary representation $(\pi,\cH)$ of $G$ 
is called {\it semibounded} if
$s_\pi$ is bounded on a non-empty open subset of $\fg$. Then 
the cone $W_\pi$, consisting of all those 
$x \in \g$ for which $s_\pi$ is bounded on a neighborhood of $x$, 
is an open $\Ad(G)$-invariant convex cone in $\g$. 
Moreover, $s_\pi \: W_\pi \to \R$ is a continuous 
convex function (\cite{Ne08}).

For a convex cone $W \subeq \g$, we say that $\pi$ is {\it $W$-semibounded} 
if $s_\pi(W) \subeq \R$ and $s_\pi \: W \to \R$ is locally bounded.

(b) A convex cone $W \subeq \g$ is said to be 
{\it relatively open} if the linear subspace 
$W - W$ of $\g$ is closed and $W$ is open in $W - W$.   
\end{de}

For the definition of a $C^1$-map we use in the next lemma see Definition~\ref{D:smooth}.
\begin{lem}\mlabel{L:diffexpcone} Let $W$ be a relatively open convex 
cone in $\fg$, $\fq:=W-W$, and $(\pi,\cH)$ 
be a smooth $W$-semibounded unitary representation of $G$. 
Then, for every $v\in\cHy(\pi)$, the map
$$ \rho^v \: W\ra \cH,\ x\mt e^{i\odpi(x)}v,$$
is $C^1$ and
$$T_x(\rho^v)(y)=
\oline{\dpi}\left(\int_0^1{e^{s\ad ix}yds}\right)e^{i\oline{\dpi}(x)}v.$$
The map 
$G \times W \to \cH, (g,x) \mapsto \pi(g)\rho^v(x)$ is also $C^1$.
\end{lem}

\begin{proof}
For $x\in W$,  $y\in\fgc$ and $v,w\in\cHy$,
the function
$$F(z):=\ps{-e^{z\odpi(x)}v,\dpi(y)w}-\ps{e^{z\odpi(x)}\dpi(e^{-z\ad x}y)v,w}$$
is continuous on the closed upper half-plane 
$\C_+ := \{ z \in \C \: \Im z \geq 0\}$ 
and holomorphic on its interior (\cite[Prop.~VI.3.2]{Ne00}; 
see also Step $2$ in the proof of Proposition~\ref{prop:4.9}). 
Since $F(t)=0$ for $t\in\bR$ by Proposition~\ref{prop:4.9}, 
the Schwarz Reflection Principle
\cite[Thm.~11.14]{RUDIN} implies that $F$ vanishes on $\bC_+$.
It follows that, for $x\in W$ and $v\in \cH^\infty$ we have 
$e^{i\oline{\dd\pi(x)}}v \in \cD\big(\oline{\dd\pi(y)}\big)$ for every 
$y \in \g$ with 
\[ \oline{\dpi}(y)e^{i\odpi(x)}v=e^{i\odpi(x)}\dpi(e^{-\ad ix}y)v.\] 
Since $\dpi(e^{-\ad ix}y)v$ is again a smooth vector, we can iterate this 
argument to obtain inductively 
\[ e^{i\odpi(x)}v\in\cD := \bigcap_{y_n,\dots,y_1\in \g,\ n\in\bN}{\D(\odpi(y_n))\dots\D(\odpi(y_1))}.\] 
We know by \cite[Lemma 3.4, Remark 8.3]{Ne10b} that $\D=\cHy$. Hence we have
\begin{equation}\label{E:comagain}
e^{i\odpi(x)}\cH^\infty\subseteq \cH^\infty
\ \text{and}\ {\dpi}(y)e^{i\odpi(x)}
=e^{i\odpi(x)}\dpi(e^{-\ad ix}y)
\end{equation} 
for $x \in W$ and $y\in\fgC$.
Now let $v\in \cH^\infty$. In view of \eqref{E:comagain}, 
Proposition~\ref{P:derpath} implies that 
$\rho^v$ is $C^1$ with 
\[ T_x(\rho^v)(y)=
{\dpi}\left(\int_0^1{e^{s\ad ix}yds}\right)
e^{i\oline{\dpi}(x)}v
=e^{i\oline{\dpi}(x)}
{\dpi}\left(\int_0^1{e^{-s\ad ix}yds}\right)v\] 
and that the map 
\[ \hat\rho \: W\times\cH\ra\cH, \quad (x,v)\mt \eidpi{x}v \] 
is continuous. 

Finally, we observe that the map 
\[ F \: G \times W \to \cH, \quad F(g,x) := \pi(g)\rho^v(x) \] 
is continuous because the $G$-action on $\cH$ defined by $\pi$ is 
continuous. We have just seen that $F$ is partially differentiable  
in $x$ with continuous partial derivative 
\[ G \times W \times \fq \to \cH, \quad 
(g,x,y) \mapsto \pi(g) T_x(\rho^v)y.\] 
We have also seen that $\rho^v(W) \subeq \cH^\infty(\pi)$, so that 
the partial derivatives in $g$ also exist and are given by 
\[ TG \times W \to \cH, \quad 
(g.y,x) \mapsto 
\pi(g) \dd\pi(y) e^{i\oline{\dd\pi}(x)}v  
= \pi(g)  e^{i\oline{\dd\pi}(x)}\dd\pi(e^{-\ad(ix)}y)v.\] 
As $G$ acts continously on $\cH$, $\hat\rho$ is continuous,  
and the adjoint action of $G$ on $\g$ is continuous, this 
function is continuous. This implies that $F$ is $C^1$ 
(cf.\ \cite{Ham82}).
\end{proof}

\begin{de} Let $W$ be a relatively open convex cone in $\fg$, 
$\fq:=W -W$, and $\fh$ be a closed subalgebra of $\fg$.
The cone $W$ is called {\it $\fh$-compatible} if
\[ [W,W]\subseteq \fh \quad \mbox{ and } \quad e^{\ad\fh}W\subseteq W.\]  
Then $\fg_c:=\fh\oplus i\fq \subeq \g_\bC$ is a closed subalgebra 
which is turned in a symmetric Banach--Lie algebra by 
the involution $\theta(x+iy) := x-iy$ for 
$x\in \fh, y \in \fq$. 

An $\fh$-compatible cone $W \subeq \g$ is 
said to be {\it integrable} if there exists a 
symmetric Banach--Lie group $(G_c,\theta)$ with 
symmetric Lie algebra $\fg_c=\fh\oplus i\fq$ such that, 
for $H := (G^\theta_c)_0$,  the polar map 
\[ H \times iW \to G_c, \quad (h,x) \mapsto h \exp(x)\] 
is an analytic diffeomorphism onto an open subsemigroup 
\[ S = S_H(iW) = H \exp(iW) \subeq G.\] 
Then $S$ is invariant under the involution 
$s^*=\theta(s)^{-1}$, turning it into an involutive semigroup 
$(S,*)$. 

From the discussion of Banach--Olshanski semigroups in 
Appendix~\ref{A}, it then follows that for each 
connected Banach--Lie group $H_1$ locally isomorphic to 
$H$ to which the adjoint action of $\fh$ on $\g_c$ integrates, 
there exists an involutive  Banach--Olshanski semigroup 
$S_{H_1}(iW)$ with a polar decomposition which is a quotient 
of the universal covering semigroup of $S_H(iW)$. 
\end{de}

\begin{theo} \mlabel{T:Analcont} Let $G$ be a Banach--Lie group with Lie 
algebra $\fg$, $\fh$ be a closed complemented 
Lie subalgebra of $\fg$, and $H := \lan \exp \fh \ran$ 
be the corresponding 
integral subgroup in $G$. Let $(\pi, \cH)$
be a smooth $W$-semibounded unitary representation of $G$ 
for the integrable $\fh$-compatible cone $W$. Then the formula
$$\rho(h\Exp ix):=\pi(h)e^{i\odpi(x)}\ \text{for}\ h\in H\ \text{and}\ x\in W,$$
defines a strongly continuous smooth $*$-representation $\rho$ of 
$S_H(iW)$ on~$\cH$. 
\end{theo}

We shall need the following Chain Rule (\cite[Lemma~4]{Mer10}): 
\begin{lem}\label{L:PR} Let $I \subeq \R$ be an open interval, 
$E$ and $F$ be two Banach spaces and
$L_s(E,F)$ denotes the space of continuous linear operators 
from $E$ to $F$ endowed
with the strong operator topology.
Let $I \to L_s(E,F), t \mapsto K(t)$ be a continuous
path such
that $t\mt K(t)v$ is differentiable for every $v$ in a subspace $\D$ of $E$ and let $\hg(t)$ be a differentiable
path in $\D$. We write $K'(t):\D\ra F$ for the linear
operator obtained by $K'(t)v:=\dif{t} K(t)v$ for $v\in\D$.
Then $t\mt K(t)\hg(t)$ is differentiable with
$$\dif{t} K(t)\hg(t)=K'(t)\hg(t)+K(t)\hg'(t).$$
\end{lem}

\begin{proof}[Proof of Theorem~\ref{T:Analcont}]
{\bf Step 1:} Let us first prove that, for $x_1,x_2\in W$,
\begin{equation}\label{E:comW}
\rho(\Exp ix_1\Exp ix_2)=\rho(\Exp ix_1)\rho(\Exp ix_2).
\end{equation}
For this purpose, let us write for $t>0$,
\begin{equation}\label{E:prod}
\eta(t) 
:= \Exp tix_1\Exp ix_2=h_t\Exp(ix(t))\ \text{with}\ h_t\in H,\ x(t)\in W.
\end{equation}
Now let $v\in\cHy$ and consider  
$\gamma(t):=\rho(\eta(t))v = \pi(h_t)e^{i\odpi(x(t))}v$. By 
Lemma~\ref{L:diffexpcone} and the Chain Rule 
(Lemma~\ref{L:PR}), applied with $K(t) = \pi(h_t)$ and 
$\cD = \cH^\infty(\pi)$, the path $\gamma(t)$ is 
differentiable for $t > 0$, and, denoting by $\delta$ the 
right logarithmic derivative (see Definition~\ref{def:logder}), we obtain 
with Proposition~\ref{P:Exp} 
\begin{align*}
\gamma'(t)
&= \oline{\dpi}(\delta h_t)\gamma(t)+\pi(h_t) 
\oline{\dd\pi}\Big(\int_0^1 e^{is\ad x(t)}ix'(t)ds\Big)
e^{\odpi(ix(t))}v\\
&=\oline{\dpi}(\delta h_t)\gamma(t)+\pi(h_t)
\oline{\dpi}\big(\delta(\Exp)_{ix(t)}ix'(t)\Big)e^{\odpi(ix(t))}v\\
&=\oline{\dd\pi}\Big(\delta(h)_t +\Ad(h_t)\delta(\Exp)_{ix(t)}ix'(t)\Big)
\gamma(t)\\
&=\oline{\dd\pi}(\delta(\eta)_t)\gamma(t)
\quad\quad\text{(by Proposition~\ref{P:pr})}\\
&=\oline{\dpi}(ix_1)\gamma(t).
\end{align*}
Since $\lim_{t\ra0}\gamma(t)
= \lim_{t\ra0} \pi(h_t)e^{i\odpi(x(t))}v =\rho(\Exp ix_2)v$ 
(Lemma~\ref{L:diffexpcone}, Lemma~\ref{lem:approx}), 
we obtain with \cite[p.~481]{Kat66} that 
\[ \gamma(t)=\rho(\Exp tix_1)\rho(\Exp ix_2)v,\] 
and \eqref{E:comW} follows for $t = 1$.

{\bf Step 2:} For $h\in H$ and $x\in W$ we have
$$\pi(h)\odpi(x)\pi(h)^{-1}=\odpi(\Ad(h)x), $$
so that 
\begin{equation}\label{E:adcom}
e^{\odpi(i\Ad(h)x)}=e^{\pi(h)i\odpi(x)\pi(h)^{-1}}=\pi(h)e^{i\odpi(x)}\pi(h)^{-1}.
\end{equation}
From \eqref{E:adcom} we obtain the relation
$$\rho(h\Exp x)^*=\rho((h\Exp x)^*),$$
and we further derive 
\begin{equation}\label{E:mult}
\rho(sh)=\rho(s)\pi(h) \quad \mbox{ for }\quad s \in S, h \in H.
\end{equation}
{\bf Step 3:} Now we can prove that for $h_1,h_2\in H$ and $x_1,x_2\in W$,
\begin{equation}\label{E:prod2}
\rho(h_1\Exp ix_1h_2\Exp ix_2)=\rho(h_1\Exp ix_1)\rho(h_2\Exp ix_2).
\end{equation}
With \eqref{E:comW} and \eqref{E:adcom} we obtain 
\begin{align*}
\rho(h_1\Exp ix_1)
\rho(h_2\Exp ix_2)
&=\pi(h_1)e^{\odpi(ix_1)}\pi(h_2)e^{i\odpi(x_2)}\\
&=\pi(h_1)\pi(h_2)e^{i\odpi(\Ad(h_2)^{-1}x_1)}e^{i\odpi(x_2)} \\ 
&=\pi(h_1h_2)\rho(\Exp(i \Ad(h_2)^{-1}x_1)) \rho(\Exp(i x_2)) \\
&=\pi(h_1h_2)\rho(\Exp(i \Ad(h_2)^{-1}x_1) \Exp(i x_2)) \\
&=\rho(h_1 h_2\Exp(i \Ad(h_2)^{-1}x_1) \Exp(i x_2)) \\
&=\rho(h_1 \Exp i x_1) h_2 \Exp i x_2).
\end{align*}

{\bf Step 4:} It remains to prove that for $v\in\cHy(\pi)$ the map
 $\rho^v:S\ra\cH$, $\rho^v(s):=\rho(s)v$ is smooth. In view of 
Lemma~\ref{L:diffexpcone}, this map is $C^1$. 
For $x\in\fh$ we have by Remark~\ref{rem:a.9} 
and \eqref{E:mult}
\[ T_s(\rho^v)(s.x)
:=\df{t}\rho(s\exp tx)v=\df{t}\rho(s)\rho(\exp tx)v=\rho(s)\dpi(x)v.\] 
Similarly we have for $x\in W$
$$T_s(\rho^v)(s.(ix)):=\df{t}\rho(s\Exp tix)v
=\df{t}\rho(s)\rho(\Exp tix)v=\rho(s)\dpi(ix)v.$$
Since  $T_s(\rho^v)$ linear, it follows that
\begin{equation}\label{E:derrho}
T_s(\rho^v)(s.x)=\rho(s)\dpi(x)v\quad \text{for}\ x\in\fh+i\fq.
\end{equation}
Now an easy induction shows that 
the higher 
partial derivatives of $T\rho^v$ only involve the continuous $n$-linear maps
\[ \omega_v^n(x_1,\dots,x_n):=\dpi(x_1)\dots\dpi(x_n)v,\]
and hence that $\rho^v$ is smooth.
\end{proof}

Now recall the context of 
the L\"uscher--Mack Theorem. 
We have a symmetric Banach--Lie algebra $\fg=\fh\oplus\fq$ and 
an integrable $e^{\ad\fh}$-invariant open convex cone $W \subeq \fq$. We
therefore have a Banach--Olshanski semigroup $S_{H}(W)=H\Exp W$ for each
connected Lie group with Lie algebra $\fh$. Applying the preceding theorem
to $\fg_c=\fh+i\fq$ and $-iW\subseteq i\fq$, we obtain the following
converse to the L\"uscher--Mack Theorem:

\begin{cor} \label{C:conv} 
Let $\fg=\fh\oplus\fq$ be a symmetric Banach--Lie algebra and 
$W$ be an integrable $e^{\ad\fh}$-invariant open convex cone in $\fq$.
Let $G_c$ be a Banach--Lie group with Lie 
algebra $\fg_c=\fh+i\fq\subseteq \fg_\bC$ and let $H_c$ be its integral subgroup
with Lie algebra $\fh$. Let $\pi$
be an $-iW$-semibounded unitary representation of $G_c$. Then 
\[ \rho(h\Exp x):=\pi(h)e^{\odpi(x)}\ \text{for}\ h\in H_c\ \text{and}
\ x\in W,\] 
defines a strongly continuous smooth $*$-representation $\rho$ of 
$S_{H_c}(W)$. 
\end{cor}

\begin{theo}[Holomorphic Extension Theorem]
Let $G$ be a Banach-Lie group with Lie 
algebra $\fg$, $(\pi, \cH)$ be a semibounded unitary representation 
of $G$, and $W\subseteq W_\pi$ be an open integrable $\Ad(G)$-invariant 
convex cone. Then 
\[ \rho(g\Exp ix)
:=\pi(g)e^{i\odpi(x)}\ \text{for}\ g\in G\ \text{and}\ x\in W,\] 
defines a holomorphic $*$-representation $\rho$ of 
the complex involutive semigroup 
\[ S_G(iW)=G\exp iW.\]
 In particular the vectors in $\rho(S_G(iW))\cH$ are 
analytic for $\pi$.
\end{theo}

Note that $S_G(iW)$ is a complex Olshanski semigroup 
(cf.\ Definition~\ref{def:olshsem}). 

\begin{proof} First we observe that the cone 
$W$ is $\fg$-compatible. 
Theorem~\ref{T:Analcont} now applies to 
the semigroup $S_G(iW)$. 
It remains to prove that $\rho:S_G(iW)\ra B(\cH)$ is holomorphic. 
But \eqref{E:derrho} shows that $T\rho^v$ is complex linear, hence
that $\rho^v$ is holomorphic. The holomorphy of $\rho$ now follows from 
{\bf \cite[Lemma~IV.2.2]{Ne00}}. 
Now let $s\in S_G(iW)$ and $v\in\cH$. 
Since $\pi^{\rho(s)v}(h)=\pi(h)\rho(s)v=\rho(hs)v$, the analyticity of
$\pi^{\rho(s)v}$ follows from the analyticity of the map $H\ra S$, $h\mt hs$. 
\end{proof}

The preceding theorem generalizes Olshanski's Holomorphic Extension Theorem 
for highest weight representations (\cite{Ol82}, \cite{Ne00}) 
to the Banach--Lie setting. 
In the finite dimensional
case the proof heavily relies on the existence of a dense space of 
analytic vectors, which can be derived  
by convolution with heat kernels (\cite{Ga60}), 
but for unitary representations of 
Banach--Lie groups, not even the space of 
$C^1$-vectors need to be dense (cf.\ \cite{Ne10b}).
In the finite dimensional context 
one proves first that $\rho$ is holomorphic, and then the multiplicativity
of $\rho$ is obtained by analytic continuation. 
In the proof we give here the multiplicativity of $\rho$ follows from 
the (assumed) existence of smooth vectors 
and then the holomorphy of $\rho$ follows as a bonus from 
its multiplicativity.

\appendix

\section{Covering theory for Olshanski semigroups}\label{A}

Let $(G,\theta)$ be a symmetric Banach--Lie group and 
$(\fg,\theta)$ the corresponding 
symmetric Lie algebra. We write 
\[ \fg=\fh\oplus \fq \quad \mbox{ with } 
\quad \fh = \ker(\theta - \1) \quad \mbox{ and } 
\quad \fq = \ker(\theta + \1),\]
for the eigenspace decomposition of $\fg$ under $\theta$ 
and let $H_G$ denote the identity component of $G^\theta$. 
We set $g^*=\theta(g)^{-1}$  
and consider an open convex $\Ad(H_G)$-invariant 
convex cone $W \subeq \fq$ 
for which the map 
\[ H_G \times W \to G, \quad (h,x) \mapsto h\exp x \] 
is an analytic diffeomorphism onto an open subsemigroup $S = H \exp(W)$ 
of $G$.  In these coordinates the involution on  $h \exp x \in S$ 
is given by 
\begin{equation}
  \label{eq:invol2}
(h \exp x)^* = (\exp x) h^{-1} = h^{-1}\exp(\Ad(h)x).
\end{equation}
In particular $S$ is $*$-invariant. 
In the following we write $S_{H_G}(W) = H_G \exp(W)$ for this involutive 
semigroup. 

\begin{de}
  \label{def:1.9} 
Let $S$ be an involutive
semigroup. A {\it multiplier} of $S$ is a pair $(\lambda, \rho)$ of maps 
$\lambda, \rho \: S \to S$ satisfying 
\[ a \lambda(b) = \rho(a)b, \qquad \lambda(ab) = \lambda(a)b, \qquad 
\mbox{ and } \quad \rho(ab) = a \rho(b) \quad \mbox{ for } \quad 
a,b \in S.\] 

We write $M(S)$ for the set of all multipliers of $S$ and turn
it into an involutive semigroup by 
\[ (\lambda, \rho) (\lambda', \rho') := (\lambda \circ \lambda', \rho'
\circ \rho)  \quad \hbox{ and } \quad 
(\lambda, \rho)^* := (\rho^*, \lambda^*), \] 
where $\lambda^*(a) := \lambda(a^*)^*$ and 
$\rho^*(a) = \rho(a^*)^*$. 
We write 
\[ \U(M(S)):= \{ (\lambda, \rho) \in M(S) \: 
(\lambda, \rho)(\lambda, \rho)^* = 
(\lambda, \rho)^*(\lambda, \rho) =\1\} \]
for the {\it unitary group} of $M(S)$. 

Note that $S \to M(S), s \mapsto (\lambda_s, \rho_s)$ is a morphism 
of involutive semigroups and that $M(S)$ acts on $S$ from the left
by $(\lambda,\rho).s := \lambda(s)$ and from the right by 
$s.(\lambda,\rho) := \rho(s)$. 
\end{de}

The group ${H_G}$ is in general 
not contained in 
$S_{H_G}(W)$, 
but it acts on it by the unitary multipliers $(\lambda_h, \rho_h)$, 
$h \in H_G$. 

\begin{pro}
Let  $q:\ti S\ra S$ be the universal covering of the Banach manifold 
$S = S_{H_G}(W)$. Then $\ti S$
carries the structure of an analytic Banach $*$-semigroup such that 
the covering map $q \: \tilde S \to S$ 
is a homomorphism of Banach $*$-semigroups. 

Moreover, the simply connected covering group $\tilde H_G$ of ${H_G}$ 
acts on $\tilde S$ by unitary multipliers and we thus obtain 
an analytic diffeomorphism 
\[  \tilde\Phi \: \tilde H_G \times W \to \tilde S, \quad 
(h,x) \mapsto h\Exp x,\] 
where $\Exp \: W \to \tilde S$ is a continuous lift of 
$\exp \: W \to S$ such that 
\[ \Exp(x)^* = \Exp(x)\quad \mbox{ for } \quad x \in W,\] 
and 
\begin{equation}
  \label{eq:Exp}
\Exp(sx)\Exp(tx) = \Exp((t+s)x) \quad \mbox{ for } \quad 
x \in W, t,s  > 0.
\end{equation}
\end{pro}

\begin{proof} Since the polar map $\Phi \: {H_G} \times W \to S$ is an 
analytic diffeomorphism, there exists an analytic diffeomorphism 
$\tilde\Phi \: \tilde H_G \times W \to \tilde S$ with 
$q \circ \tilde\Phi = \Phi \circ q$. 
We then define $\tilde\Exp \: W \to \tilde S, x \mapsto \tilde\Phi(e,x)$. 

Pick $x_0 \in W$ and let $\tilde m \: \tilde S \times \tilde S \to \tilde S, 
(s,t) \mapsto st$ be the 
unique continuous lift of the multiplication map 
$m \: S \times S \to S$ with 
\[ \tilde\Exp(x_0)\tilde\Exp(x_0) = \tilde\Exp(2x_0).\] 
Then the uniqueness of lifts implies that the 
restriction of $\tilde m$ to 
$\tilde\Exp(\R_{>0} x_0)$ satisfies 
\[ \tilde\Exp(tx_0)\tilde\Exp(sx_0) = \tilde\Exp((t+s)x_0)\] 
and we obtain in particular 
\[ (\tilde\Exp(x_0)\tilde\Exp(x_0))\tilde\Exp(x_0) 
= \tilde\Exp(x_0)(\tilde\Exp(x_0)\tilde\Exp(x_0)).\] 
Therefore the uniqueness of lifts implies that $\tilde m$ is 
associative, hence defines on $\tilde S$ an analytic semigroup 
structure. 

We also lift the involution on $S$ to the unique involutive
diffeomorphism $*$ on $\tilde S$ with 
$\tilde\Exp(x_0)^* = \tilde\Exp(x_0)$, and 
since $(st)^* = t^*s^*$ now holds for 
$s = t = \tilde\Exp(x_0)$, the uniqueness of lifts implies that 
$(\tilde S, *)$ is an involutive semigroup. 

The multiplication on $S$ can be expressed by 
analytic maps $m_{H_G} \: W \times W \to {H_G}$ and 
$m_W \: W \times W \to W$ as 
\begin{equation}
  \label{eq:mult}
\Phi(h,x)\Phi(h',x') 
= \Phi(hh' m_{H_G}(\Ad(h')^{-1}x,x'), m_W(\Ad(h')^{-1}x,x')).
\end{equation}
From the continuity of the multiplication 
$S\times  (S \cup H_G) \to S$ in $G$, it follows that 
both maps $m_W$ and $m_H$ extends continuously to 
the set 
\[ W_2 := \big(W \times (W \cup \{0\})\big) 
 \cup \big((W \cup \{0\}) \times W\big).\] 

If $\tilde m_{H_G} \: W_2 \to \tilde H_G$ is the unique lift of $m_{H_G}$ 
satisfying $\tilde m_{H_G}(x_0, x_0) = e$, then we obtain the formula 
\begin{equation}
  \label{eq:pro}
\tilde\Phi(h,x)\tilde\Phi(h',x') 
= \tilde\Phi(hh' \tilde m_{H_G}(\Ad(h')^{-1}x,x'), m_W(\Ad(h')^{-1}x,x')).
\end{equation}
For each $x \in W$ and $t,s > 0$ we further have 
$\tilde m_{H_G}(sx,tx) = e$ because \break 
$m_{H_G}(sx,tx) = e$ and the subset 
$\{ (sx,tx) \in W \times W \: s,t>0, x \in W \}$ is connected. 
This implies \eqref{eq:Exp}. 

The left and right multiplier actions of ${H_G}$ on $S$ lift 
to unique left and right actions of $\tilde H_G$ on $\tilde S$, 
satisfying 
\begin{equation}
  \label{eq:hact}
h\tilde\Phi(h',x) = \tilde\Phi(hh',x) '\quad \mbox{ and } \quad 
\tilde\Phi(h',x)h = \tilde\Phi(h'h,\Ad(h)^{-1}x).
\end{equation}
From \eqref{eq:hact} we further derive 
\begin{equation}
  \label{eq:mult2}
(h\tilde\Exp x)(h'\tilde\Exp x') = hh'\tilde\Exp(\Ad(h')^{-1}x)\tilde\Exp(x').
\end{equation}
This implies in particular that the left action of 
$\tilde H_G$ on $\tilde S$ commutes with the right multiplications. 
We also obtain from the uniqueness of lifts that  
\[ (h\tilde\Exp x)^* = h^{-1}\tilde\Exp(\Ad(h)x)= \tilde\Exp(x)h^{-1} 
\quad \mbox{ for } \quad x \in W, h \in \tilde H_G, \]  
so that 
left multiplications in $\tilde S$ commute with the right action of 
$\tilde H_G$. 
To see that $\tilde H_G$ acts on $\tilde S$ by unitary multipliers, 
it remains to observe that 
\[ (h'\tilde\Exp(x')h)(h''\tilde\Exp(x'')) 
=  (h'\tilde\Exp(x'))(hh''\tilde\Exp(x''))\] 
for $h,h',h'' \in \tilde H_G, x,x'\in W$, 
which also follows from \eqref{eq:mult2}.
\end{proof}

Let $\Ad_\fq^{\tilde H_G} := \Ad_\fq \circ q_{H_G}$ be the action of 
$\tilde H_G$ on $\fq$, obtained from the action $\Ad_\fq$ of ${H_G}$ on 
$\fq$ and the covering map $q_{H_G} \: \tilde H_G \to {H_G}$. 

\begin{pro} For a discrete central subgroup $\Gamma \subeq \tilde H_G$ 
acting trivially on $\fq$, the cosets in $\tilde S$ satisfy 
\begin{equation}
(s\Gamma)(t\Gamma) = st\Gamma \quad \mbox{ and } \quad 
(s\Gamma)^* = s^*\Gamma \quad \mbox{ for } \quad 
s,t \in \tilde S,
  \label{eq:congru}
\end{equation}
so that the quotient semigroup $\tilde S/\Gamma$ 
inherits the structure of an 
analytic involutive Banach semigroup for which the quotient map 
$q_\Gamma \: \tilde S \to \tilde S/\Gamma$ is a morphism of 
involutive semigroups. Moreover, 
the polar map $\tilde\Phi \: \tilde H_G \times W \to \tilde S$ 
factors through a diffeomorphism 
$\Phi_\Gamma \: \tilde H_G/\Gamma \times W \to \tilde S/\Gamma$ 
and the group $\tilde H_G/\Gamma$ acts faithfully on $\tilde S/\Gamma$ 
by unitary multipliers. 
\end{pro} 

\begin{proof} Since the left action of $\Gamma$ on $\tilde S$ coincides 
with the right action (see \eqref{eq:hact}), 
the relations \eqref{eq:congru} easily follow. 
The remaining assertions are now obvious.
\end{proof}

\begin{de} \mlabel{def:olshsem} 
(a) The semigroups obtained by the preceding proposition
will be called Banach--Olshanski semigroups. We write 
\[ S_{H}(W):=\ti S_{H_G}(W)/\Gamma 
\quad \mbox{ for } \quad 
H = \tilde H_G/\Gamma\]  
and $\exp x := \tilde\exp(x) \cdot\Gamma$ for the exponential 
function $\exp \: W \to S_H(W)$. 

(b) For the special case where 
$\g =\fh_\C$ and $\theta(x + iy) = x-iy$, the group $G$ is complex, 
so that $\Gamma_{H_G}(W)$ is a complex manifold on which 
the multiplication is holomorphic and the involution is 
antiholomorphic. These properties are inherited by all other 
Olshanski semigroups $S_H(W)$, where $H \cong \tilde H_G/\Gamma$ 
is a connected Lie group with Lie algebra $\fh$. 
Therefore we call them {\it complex Olshanski semigroups}. 
\end{de}

\begin{rem} \mlabel{rem:a.5} The basic properties of Olshanski semigroups 
$S_H(W)$ are: 
\begin{description}
\item[\rm(O1)] The polar map 
$H \times W \to S_H(W), (h,x) \mapsto h \Exp x$ is an analytic diffeomorphism.
\item[\rm(O2)] $H$ acts on $S_H(W)$ smoothly by unitary multipliers. 
\item[\rm(O3)] For $x \in W$, we have 
\[ \exp(sx)\exp(tx) = \exp((t+s)x)\quad \mbox{ for } \quad t,s > 0.\]
\item[\rm(O4)] For $h \in H$ and $s \in W$, we have 
$(h\exp x)^* = (\exp x)h$. 
\end{description}
\end{rem}

From now on $H$ always denotes a connected Lie group with 
Lie algebra $\fh$ and $S = S_H(W)$ is a corresponding 
Olshanski semigroup. We do not assume that $H$ is contained in $G$. 

\begin{lem} \mlabel{lem:approx} 
For $s \in S_H(W)$ and $x \in W$ we have
\[ \lim_{t \to 0_+} \Exp(tx)s =  \lim_{t \to 0_+} s\Exp(tx) = s.\] 
\end{lem} 

\begin{prf} It suffices to verify this relation in the simply 
connected covering semigroup $\tilde S$, where we have for 
$s = h'\Exp(x')$: 
\[ \Exp(tx)s 
= \tilde\Phi(e,tx)\tilde\Phi(h',x') 
= \tilde\Phi(h' \tilde m_H(\Ad(h')^{-1}tx,x'), m_W(\Ad(h')^{-1}tx,x')).\]  

Since the functions $\tilde m_H$ and $m_W$ extend continuously to the 
domain $W_2$, formula \eqref{eq:pro} yields 
$\Exp(tx)s \to s$. 
The other relation is obtained by applying the involution~$*$.
\end{prf}

\begin{rem} Let $\kappa^r \in \Omega^1(G,\fg)$ denote the {\it right 
Maurer--Cartan form}, defined by 
$\kappa^r_g(x.g) := x$ for $x \in \g = T_e(G)$, 
where $TG \times G \to TG, (v,g) \mapsto v.g$ denotes the 
canonical right action of $G$ on $TG$. Similarly we define the 
{\it left Maurer--Cartan form} by $\kappa^l_g(g.x) := x$. 

For the subsemigroup $S = S_{H_G}(W) \subeq G$, 
the restriction $\kappa^r_S := \kappa^r\res_S$ defines a 
trivialization of the tangent bundle of $S$ by 
\[ TS \to S \times \g, \quad v_s \mapsto (s,\kappa^r_S(v_s)) \quad 
\mbox{ for } \quad v_s \in T_s(S). \] 
If $q_S \: \tilde S \to S$ is the universal covering, the 
form 
$\kappa^r_{\tilde S} := q_S^*\kappa^r_S$ likewise trivializes $T(\tilde S)$. 
For every discrete central subgroup $\Gamma \subeq \tilde H_G$ acting 
trivially on $\fq$ and $S = S_{H_G}(W)$, the form $\kappa^r_{\tilde S}$ on $\tilde S$ 
is $\Gamma$-invariant, hence is the pullback of a form 
$\kappa^r_{\tilde S/\Gamma}$ trivializing $T(\tilde S/\Gamma)$. 
\end{rem}

\begin{de} \label{D:tangiso} The preceding discussion shows that on every 
Olshanski semigroup $S_H(W)$, we have a natural form 
$\kappa^r_S \in \Omega^1(S,\fg)$ trivializing the tangent bundle 
and we similarly obtain a left invariant form 
$\kappa^l_S$. 

Accordingly, we have natural {\it left invariant vector fields} 
$V_x$, $x \in \g$, on $S$, defined by $\kappa_S^l(V_x) = x$. 
and {\it right invariant vector fields} 
$W_x$, $x \in \g$, defined by $\kappa_S^r(W_x) = x$. 
\end{de}

\begin{rem}\mlabel{rem:a.9} For $x \in \fh\cup W$ and $s \in S$ we have 
\[ V_x(s) = \derat0 s\exp(tx) 
\quad \mbox{ and } \quad 
W_x(s) = \derat0 \exp(tx)s.\] 
Both relations are obvious for the semigroup 
$S_{H_G}(W)$, and they are inherited by the simply connected 
covering and hence also by its quotients. 

For $x \in \g$ and $s \in S$, we write 
$t \mapsto s\exp(tx)$ for the integral curve of 
$V_x$ through $s$ and likewise 
$t \mapsto \exp(tx)s$ for the local integral curve 
of $W_x$. This is redundant for $S = S_{H_G}(W) \subeq G$, and 
for a general $S = S_H(W)$, the preceding observation shows that 
it is also consistent for 
$x \in W\cup \fh$ with the action of the corresponding 
one-parameter (semi)groups.
\end{rem}

\begin{rem} \mlabel{rem:rightinv} 
(a) From the right invariance of $\kappa^r_G$ on $G$, 
we obtain 
\[ \rho_s^*\kappa^r_S = \kappa^r_S \quad \mbox{ for } \quad  s \in S = S_H(W)\] 
by verifying that this property is preserved by the passage 
to covering semigroups and to quotients by discrete central 
subgroups. We likewise get 
\[ \lambda_s^*\kappa^l_S = \kappa^l_S \quad \mbox{ for } \quad  s \in S.\] 

(b) As in (a), it follows that, for $h \in H$, 
the right multiplication $\rho_h \: S \to S$ 
also leaves $\kappa^r_S$ invariant. For the left multiplication 
$\lambda_h(s) = h.s$, the relation 
$\lambda_g^*\kappa^r_G = \Ad(g)\circ \kappa^r_G$ for the Maurer--Cartan 
form of a Lie group $G$ implies that 
\begin{equation}\label{eq:trafo}
\lambda_h^*\kappa^r_S = \Ad(h)\circ \kappa^r_S\quad \mbox{ for } \quad 
h \in H.
\end{equation}
This is also verified by the passage through the universal covering 
semigroup. 
\end{rem}

\begin{de} \mlabel{def:logder} 
For a smooth map $f \: M \to S = S_H(W)$, where 
$M$ is a smooth manifold, we define the (right) 
logarithmic derivative as 
the $\g$-valued $1$-form 
\[ \delta(f) := f^*\kappa_S^r \in \Omega^1(M,\g).\] 
If $I \subeq \R$ is an interval and 
$\alpha \: I \to S$ a differentiable path, then 
the identification of $1$-forms on $I$ with $\g$-valued functions leads to 
\[ \delta(\alpha)_t := \kappa^r_{\alpha(t)}(\alpha'(t)) \in \g. \] 
We likewise define logarithmic derivatives for maps with values 
in Lie groups. 
\end{de}

For the exponential map $\exp \: \g \to G$ we then have 
\[ \delta(\exp)_x(y) =  \int_0^1 e^{s\ad x}y\, ds,\]
(\cite[Prop.~II.5.7]{Ne06}) and therefore: 

\begin{pro}\label{P:Exp}
Any $C^1$-path $\ha:I\ra W$ satisfies 
\[ \delta(\Exp\ha)_t=\int_0^1{e^{s\ad\ha(t)}\ha'(t)ds}.\] 
\end{pro}

From the differential of the multiplications, we obtain left and right 
actions 
\[ S \times TS \to TS, \quad (s,v) \mapsto s.v, \quad 
 TS \times S \to TS, \quad (v,s) \mapsto v.s\] 
and likewise 
\[ H \times TS \to TS, \quad (h,v) \mapsto h.v, \quad 
 TS \times H \to TS, \quad (v,h) \mapsto v.h\]  
as well as 
\[ TH \times S \to TS, \quad (v,s) \mapsto v.s, \quad 
 S \times TH \to TS, \quad (s,v) \mapsto s.v.\]  
We then also have 
\begin{equation} \label{eq:kapparel}
(\kappa^r_S)(v.s) = v \quad \mbox{ for } \quad 
v \in \fh = T_e(H), s \in S 
\end{equation}
because this is true for the subsemigroup $\Gamma_{H_G}(W)$ of $G$. 

\begin{pro}\label{P:pr} Let $\ha:I\ra H$ and $\hb:I\ra W$ be two $C^1$-paths.
For the path $\gamma(t) := \alpha(t)\Exp(\beta(t))$ in $S$ we then have 
\[ \delta(\gamma)_t=\delta(\ha)_t+\Ad(\ha(t))\delta(\Exp \circ \beta)_t.\]
\end{pro}

\begin{proof} First we note that 
\[ \gamma'(t) = \ha'(t).\Exp\hb(t)+\ha(t).(\Exp \circ \beta)'(t).\]
To evaluate $\delta(\gamma)_t = (\kappa^r_S)_{\gamma(t)}(\gamma'(t))$, 
we first write $\alpha'(t) = \delta(\alpha)_t.\alpha(t)$ in $TH$. 
For $v \in TH$, the relation 
$(h'h).s = h'.(h.s)$ for $s \in S, h,h' \in H$ leads to 
$(v.h).s = v.(h.s),$ 
so that \eqref{eq:kapparel} implies that 
$\kappa^r_S(\ha'(t).\Exp\hb(t)) = \delta(\alpha)_t$. 
Finally, the relation \eqref{eq:trafo} leads to 
\[ \kappa^r_S\Big(\ha(t).(\Exp \circ\hb)'(t)\Big) 
= \Ad(\alpha(t))\delta(\Exp \circ \beta)_t.\qedhere\]  
\end{proof}

\section{Families of one-parameter semigroups}\label{B}

In this section we let $\fg=(\fg,[\cdot,\cdot])$ be a locally convex Lie algebra. This means 
that $\fg$ is a locally convex space
and the Lie bracket $[\cdot,\cdot]$ is continuous. Let us first recall the following
basic notions of the differential calculus over locally convex spaces.

\begin{de}
(a) Let $E,F$ be two locally convex spaces, $U$ open and 
$f:U\subseteq E\ra F$ be a continuous map on the open set $U$ of $E$. Then $f$ is called $C^1$
if the directional derivatives 
$\partial_vf(x):=\lim_{h\ra 0}\frac{f(x+hv)-f(x)}{h}$ exist for every
$x\in U$ and $v\in E$ and the map 
$$\dd f:U\times E\ra F,\quad (x,v)\mt \partial_vf(x)$$
is continuous.\\
(b) A continuous map $f:U\subseteq E\ra V$ is called $C^k$, $k\geq 2$ if it is $C^1$
and $\dd f$ is $C^{k-1}$. It is called $C^\infty$, or \emph{smooth}, if it is $C^k$ for every $k\in\bN$.\\
(c) A locally convex space $E$ is called \emph{Mackey complete}
if for each smooth curve $\xi:[0,1]\ra E$ the weak integral 
$\int_0^1\xi(t)dt$ exists, i.e., there exists a (unique) element 
$I =: \int_0^1\xi(t)dt \in E$ satisfying 
\[ \alpha(I) = \int_0^1 \alpha(\xi(t))\, dt 
\quad \mbox{ for each } \quad \alpha \in E'.\] 
This implies in particular that 
the curve $\eta(s) := \int_0^s\xi(t)dt$ is smooth 
and satisfies $\eta'=\xi$.
\end{de}

\begin{de}\label{D:smooth}
A locally convex Lie algebra $\fg$ is called $\ad$-\emph{integrable}
if for every $x\in\fg$ the (linear) vector field defined by 
$\ad x$ is complete, that is, if there exists a smooth map
$\Phi^x:\bR\times\fg\ra\fg$ with
$\df{t}\Phi^x(t)y=\ad x(y)$. We will then use the notation $e^{t\ad x}:=\Phi^x(t)$.
\end{de}

We consider a linear homomorphism
$$\ha:\fg\ra\End(\D)$$ 
in the space of endomorphism
of a dense domain $\cD$ of a Banach space $E$. We assume
that $\ha$ is strongly continuous in the sense that for every
$v\in\cD$ the map $\ha^v:\fg\ra E$, $x\mt\ha(x)v$ is continuous.
Now let $W$ be a convex cone in $\g$ which is relatively open in 
its span $\fq:=W-W$ (cf.\ Definition~\ref{D:5.3}), 
and assume that,
for every $x\in W$, the closure $\ov{\ha(x)}$ of $\ha(x)$
generates a strongly continuous semigroup 
$\Big(e^{t\ov{\ha(x)}}\Big)_{t\geq 0}$ 
and that the map $s_\ha(x):=\sup_{0\leq t\leq 1}\nm{e^{t\ov{\ha(x)}}}$ is 
locally bounded on $W$.

\begin{rem} If $E=\cH$ is a Hilbert and for every $x\in W$ the semigroup
$\Big(e^{t\ov{\ha(x)}}\Big)_{t\geq 
0}$ consists of normal operators, then the fact it generates  
a commutative $C^*$-algebra isomorphic to some 
$C(X)$ implies that 
\[ \nm{e^{t\ov{\ha(x)}}}=\nm{e^{\ov{\ha(x)}}}^t\quad \text{for}\ x\in W,\] 
and hence that $s_\ha(x)=\max\Big\{1,\nm{e^{\ov{\ha(x)}}}\Big\}$.
\end{rem}

We will need the following lemma (see \cite[Lemma 9]{Mer10}):
\begin{lem}\label{L:diff}
Consider two operators $A$ and $B$ 
defined on a common dense domain $\D$ of the Banach space $E$ and whose 
closures generates strongly continuous semigroups 
$\Big(e^{t\ov{A}}\Big)_{t\geq0}$
and $\Big(e^{t\ov{B}}\Big)_{t\geq0}$ respectively.  
Assume further that $e^{sA}\D\subseteq\D$ 
for all $s\geq 0$. If, 
for some $v\in\D$ the map $s\mt Be^{sA}v$
is continuous on $[0,\infty)$, then 
\[ e^{t\ov B}v-e^{t\ov A}v=\int_0^t{e^{s\ov B}(B-A)e^{(t-s)\ov A}vds}
\quad \mbox{ for } \quad t \geq 0.\] 
\end{lem}
We then have the following proposition generalizing 
\cite[Prop. 10]{Mer10}:

\begin{pro}\label{P:derpath}Let $\fg$ be an $\ad$-integrable Mackey complete locally convex
Lie algebra. If, for every $x\in W$ and every $y\in\fg$, 
\begin{equation}\label{E:localcomrel}
e^{\ov{\ha(x)}}\cD\subseteq\cD\quad\text{and}\quad\ha(y)e^{\ov{\ha(x)}}=
e^{\ov{\ha(x)}}\ha(e^{-\ad x}y), \end{equation}
then the map 
$$\hat{\rho}:W\times E\ra E,\ (x,v)\mt \rho^v(x):=e^{\ov{\ha(x)}}v$$
is continuous, and for every $v\in\cD$, $\rho^v$ is $C^1$ with
\[ T_x(\rho^v)(y)=
\ha\Big(\int_0^1{e^{s\ad x}yds}\Big)e^{\ov{\ha(x)}}v=
e^{\ov{\ha(x)}}\ha\Big(\int_0^1{e^{-s\ad x}yds}\Big)v.\] 
\end{pro}

\begin{prf} Let $v\in\D$ and $x,y\in\fg$. 
From the relation 
\begin{equation}\label{E:estimb}
\ha(y)e^{s\ov{\ha(x)}}v
=e^{s\ov{\ha(x)}}\ha(e^{-s\ad{x}}y)v,
\end{equation}
the continuity of the map 
$[0,\infty[\times E\to E, (s,v)\mt e^{s\ov{\ha(x)}}v$ 
and the strong continuity of $\alpha$, we derive that 
$s\mt\ha(y)e^{s\ov{\ha(x)}}v$ is continuous on $[0,\infty[$.
We can therefore apply Lemma~\ref{L:diff} to obtain 
\begin{align}
&e^{\ov{\ha(y)}}v-e^{\ov{\ha(x)}}v=\label{E:intf}
\int_0^1{e^{u\ov{\ha(y)}}\ha(y-x)e^{(1-u)\ov{\ha(x)}}vdu}\\
&=\int_0^1{e^{u\ov{\ha(y)}}e^{(1-u)\notag
\ov{\ha(x)}}\ha\big(e^{(u-1)\ad x}(y-x)\big)vdu}.
\end{align}
Let $\he>0$ let $\V$ be a $0$-neighbourhood in $\fg$ such that
for every $y\in x+\V$, $s_\ha(y)<M$. Consider now the continuous map
$$F:[0,1]\times\fg\ra E,\ (u,z)\mt\ha^v\big(e^{(u-1)\ad x}z\big).$$
Since $F([0,1]\times \{0\}) = \{0\}$, 
the compactness of $[0,1]$ implies the existence 
of a $0$-neighborhood $\V'\subeq \g$ such that 
$F([0,1] \times \V')$ is contained in the open 
ball $B(0,\frac{M}{\he})$ of radius $\frac{M}{\he}$ around $0$ in $E$. 
Thus, for every $y\in x+\V\cap\V'$, 
$\left|e^{\ov{\ha(y)}}v-e^{\ov{\ha(x)}}v\right|\leq\he$,
and this proves that $\rho^v$ is continuous.
The local boundedness of 
$s_\alpha \: W \to \R$ further implies that
$\rho^v$ is continuous for every $v\in E$, and hence the map
\[ \hat{\rho}:W\times E\ra E,\ (x,v)\mt e^{\ov{\ha(x)}}v\] 
is continuous.

Let $y\in\fq$ and let $\tau>0$ such that $x+hy\in W$ for $|h|<\tau$. We derive 
from \eqref{E:intf} the formula 
\[ \frac{e^{\ov{\ha(x+hy)}}v-e^{\ov{\ha(x)}}v}{h}=
\int_0^1{e^{s\ov{\ha(x+hy)}}\ha(y)e^{(1-s)\ov
{\ha(x)}}
vds}.\] 
Let us fix $0<\he\leq1$. Then the continuity of the map
\[ (s,h)\mt e^{s\ov{\ha(x+hy)}}\ha(y)e^{(1-s)\ov{\ha(x)}}v\] 
on $[\he,1]\times[-\tau,\tau]$ implies that we can pass to the 
limit under the integral sign to derive
\begin{equation}\label{E:lim}
\lim_{h\ra0}\int_\he^1{e^{s\ov{\ha(x+hy)}}\ha(y)e^{(1-s)\ov
{\ha(x)}}vds}=\int_\he^1{e^{s\ov{\ha(x)}}\ha(y)e^{(1-s)
\ov{\ha(x)}}vds}.
\end{equation}
The same type of argument as the one used for the continuity of $\rho^v$
shows that the integrand of
$$\int_0^\he{e^{s\ov{\ha(x+hy)}}\ha(y)e^{(1-s)
\ov{\ha(x)}}vds}$$
is bounded unformly with respect to $(s,h)$, and hence the integral is uniformly small
(with respect to $h$) when $\he$ is sufficently close to $0$.
Therefore
\begin{align*}
(\partial_y\rho^v)(x)&=\int_0^1{e^{s\ov{\ha(x)}}\ha(y)e^{(1-s)
\ov{\ha(x)}}vds}=\int_0^1{\ha(e^{s\ad{x}}y)
e^{\ov{\ha(x)}}vds}\\
&=\ha\Big(\int_0^1{e^{s\ad{x}}yds}\Big)
e^{\ov{\ha(x)}}v,
\end{align*}
where the last equality follows from the uniqueness of the weak integral.
Similarly we obtain
$$\partial_y\rho^v(x)=e^{\ov{\ha(x)}}\ha\Big(\int_0^1{e^{-s\ad{x}}yds}\Big)v,$$
and now the continuity of $\hat\rho$ implies that
$T\rho^v:W\times\fq\ra E$
is continuous, i.e, that $\rho^v$ is a $C^1$-map. 
\end{prf}

\bibliographystyle{amsalpha}

\providecommand{\bysame}{\leavevmode\hbox to3em{\hrulefill}\thinspace}
\providecommand{\MR}{\relax\ifhmode\unskip\space\fi MR }
\providecommand{\MRhref}[2]{%
  \href{http://www.ams.org/mathscinet-getitem?mr=#1}{#2}
}
\providecommand{\href}[2]{#2}

\end{document}